\documentclass{article}
\usepackage[utf8]{inputenc}
\usepackage{amsmath}
\usepackage{amssymb}
\usepackage{mathtools}
\usepackage{amssymb}
\usepackage{amsthm}
\usepackage{titlesec}
\usepackage{float}
\usepackage{graphicx}
\usepackage{ifpdf}
\usepackage{tikz}
\usetikzlibrary{patterns}
\usepackage{color}
\usepackage{bbding}
\usepackage{wrapfig}
\usepackage{mathrsfs}
\usepackage{pst-node}
\usepackage[a4paper, total={6in, 8in}]{geometry}
\usepackage[nottoc,notlot,notlof]{tocbibind}
\usepackage{tikz-cd} 
\tikzcdset{scale cd/.style={every label/.append style={scale=#1},
    cells={nodes={scale=#1}}}}
\usepackage{sectsty}
\sectionfont{\centering}

\makeatletter
\@addtoreset{equation}{section}
\makeatother

\usepackage[linktocpage]{hyperref}
\hypersetup{hypertex=true,
            colorlinks=true,
            linkcolor=blue,
            anchorcolor=blue,
            citecolor=blue}

\newtheorem{theorem}{Theorem}[section]
\newtheorem{lemma}[theorem]{Lemma}
\newtheorem{proposition}[theorem]{Proposition}
\newtheorem{corollary}[theorem]{Corollary}

\theoremstyle{definition}
\newtheorem{definition}[theorem]{Definition}

\newtheorem{assumption}[theorem]{Assumption}

\theoremstyle{remark}
\newtheorem{remark}[theorem]{Remark}

\numberwithin{equation}{section}
\numberwithin{figure}{section}

\newcommand\vir{\mathrm{vir}}

\title{\textbf{GROSS-SIEBERT INTRINSIC MIRROR RING FOR SMOOTH LOG CALABI-YAU PAIRS}}
\author{YU WANG}
\date{}

\begin{document}
\maketitle
\begin{abstract}
    In this paper, we exhibit a formula relating punctured Gromov-Witten invariants used by Gross and Siebert in \cite{GS2} to 2-point relative/logarithmic Gromov-Witten invariants with one point-constraint for any smooth log Calabi-Yau pair $(W,D)$. Denote by $N_{a,b}$ the number of rational curves in $W$
meeting $D$ in two points, one with contact order $a$ and one with
contact order $b$ with a point constraint. (Such numbers are 
defined within relative or logarithmic Gromov-Witten theory).
We then apply a modified version of deformation to the normal cone technique and the degeneration formula developed in \cite{KLR} and \cite{ACGS1} to give a full understanding of $N_{e-1,1}$ with $D$ nef where $e$ is the intersection number of $D$ and a chosen curve class. Later, by means of punctured invariants as auxiliary invariants, we prove, for the projective plane with an elliptic curve $(\mathbb{P}^2, D)$, that all standard 2-pointed, degree $d$, relative invariants with a point condition, for each $d$, can be determined by exactly one of these degree $d$ invariants, namely $N_{3d-1,1}$, plus those lower degree invariants. 
    In the last section, we give full calculations of 2-pointed, degree 2, one-point-constrained relative Gromov-Witten invariants for $(\mathbb{P}^2, D)$.
    
\end{abstract}
\tableofcontents

\addtocontents{toc}{\protect\setcounter{tocdepth}{2}}
\addcontentsline{toc}{section}{References}

\section{Introduction}

In \cite{GS2}, Gross and Siebert associated to a simple normal crossings
log Calabi-Yau pair $(W,D)$ a ring, $R(W,D)$. In the case that $D$
is maximally degenerate, i.e., has a zero-dimensional stratum, the 
expectation is that $R(W,D)$ is the coordinate ring of the mirror to
the pair $(W,D)$. Here we consider a very different case, namely,
the situation where $D$ is a smooth anti-canonical divisor on $W$.
The first interesting case is when $W=\mathbb{P}^2$ and $D$ is a smooth
cubic curve.

The product rule on the ring $R(W,D)$ is defined using \emph{punctured
invariants}, introduced in \cite{ACGS2}. There are a generalization of the
logarithmic Gromov-Witten invariants introduced in \cite{GS1}, \cite{Ch}, 
\cite{AC}. Punctured invariants allow maps with negative contact order
with the divisor $D$. In general, they may be very difficult to calculate
and as yet not many techniques for their calculation have been developed.

One of the main results of this paper gives a relationship between
the punctured invariants necessary to define the ring $R(W,D)$
and logarithmic Gromov-Witten invariants. These logarithmic
invariants in turn have been proven by Abramovich, Marcus and Wise
in \cite{AMW} to coincide with Jun Li's relative invariants. Thus,
once this relationship is established, we obtain a description of
$R(W,D)$ in terms of relative Gromov-Witten invariants.

Punctured Gromov-Witten invariants appear in the Gross-Siebert mirror construction for a general log Calabi-Yau pair as the structure coefficients of the coordinate ring of the mirror degeneration, with the space of non-negative contact orders representing generators. The whole construction has been studied in \cite{GS2}. Roughly speaking, the relevant punctured Gromov-Witten invariants are defined using the moduli space of stable punctured logarithmic maps with 2 marked points, and exactly one non-trivial punctured point with a point constraint
at this latter point.

In order to obtain the initial results, we will need to
 studying the splitting and gluing behavior of those moduli spaces relevant 
to the construction of $R(W,D)$.
In order to speak of splitting and gluing phenomena properly, \cite{ACGS2}
stratifies moduli spaces of punctured maps by tropical types.
Morally speaking, for any punctured logarithmic map $f:C/T\rightarrow W$ over a log scheme $T$, there exists a functorial way associating to it a family of maps in the category of generalized cone complexes. This
process is called \textit{tropicalization}, and we write $\Sigma(f):\Sigma(C)/\Sigma(T)\rightarrow\Sigma(W)$ for the corresponding tropical map. 
Associated to this tropical map is certain combinatorial data, the
\emph{tropical type} of the map, which we denote by
$\pmb{\tau}$. There exists moduli spaces $\mathscr{M}(W,\pmb{\tau})$ of
punctured maps marked by the tropical type $\pmb{\tau}$, and these give
a stratification of the corresponding moduli space $\mathscr{M}(W,\pmb{\beta})$ of punctured log maps of type $\pmb{\beta}$.

In the context of the construction of $R(W,D)$, the type $\pmb{\beta}$
indicates a curve class of the punctured map and contact orders $p,q$ and
$-r$ for the three marked or punctured points on the domain curve.
We then have variant moduli spaces $\mathscr{M}(W,\pmb{\beta},z)$ 
and $\mathscr{M}(W,\pmb{\tau},z)$ after imposing requiring that the punctured
point with contact order $-r$ maps to a suitably chosen point $z\in W$.
See \S\ref{subsec:relevant} for more details.

In the case that $D$ is smooth, the ring $R(W,D)$ can now be defined
as follows. First, any non-negative contact order is an integral point
of $\Sigma(W)=\mathbb{R}_{\ge 0}$, i.e., all contact orders lie in $\mathbb{N}$.
Choose a finitely generated, saturated submonoid $P\subseteq H_2(W,\mathbb{Z})$
containing all effective curve classes on $W$. Then we have (in the case
that $D$ is ample)
\[
R(W,D):=\bigoplus_{p\in \mathbb{N}} k[P]\vartheta_p
\]
where the $\vartheta_p$ are generators. We then
define the product structure via the formula
\[
\vartheta_p\cdot \vartheta_q = \sum_{\beta\in P}
\sum_{r\in \mathbb{N}} N^{\pmb{\beta}}_{pqr}
t^{\beta}\vartheta_r
\]
where 
\[
N^{\pmb{\beta}}_{pqr} = \deg [\mathscr{M}(W,\pmb{\beta},z)]^{\mathrm{virt}}.
\]
Here $\pmb{\beta}$ is the punctured type with curve class $\beta$
and contact orders $p,q,-r$ at the three marked or punctured points.

If $D$ is not ample, then the above product rule may not be a finite sum,
and in this case one replaces $k[P]$ with various Artinian
quotients of $k[P]$ to avoid this problem. See \cite{GS2} for details
on this point.

Then, by exploring gluing and splitting, we obtain (see Corollary
\ref{corollary 3.2.3}):

\begin{theorem}\label{theorem 1.1}
Let $(W,D)$ be as above, $p,q,r\in \mathbb{N}$, $r>0$. Then we have
$$N_{pqr}^{\pmb{\beta}}=(q-r)N_{p,q-r}+(p-r)N_{q,p-r},$$
where $N_{a,b}$ is the logarithmic or relative Gromov-Witten invariant
counting two-pointed rational curves with contact orders $a,b$ with
$D$, and satisfying a point constraint at the second point.
\end{theorem}

Next, set $e=\beta\cdot D$, then we realize that the invariants $N_{e-1,1}$ for any smooth log Calabi-Yau pair $(W,D)$ with an extra hypothesis for D being nef have a close relation with closed Gromov-Witten invariants $n_{\beta+fh}$ defined and studied in \cite{Cha}, and further investigated in \cite{LLW} and \cite{Lau}. Their closed GW invariants involve a moduli space 
$\mathscr{M}_{0,1}(X,\beta+h, s)$, that is, the moduli space of genus 0, 1-marked relative stable maps to $X=\mathbb{P}(\omega_W\oplus\mathcal{O})$ with the curve class $\beta+h$ passing through a fixed point $s$ in $X$. To compare Chan's invariants
with the numbers we require, we use a slightly modified version of 
deformation to the normal cone and degeneration formula, following
a similar strategy as in \cite{vGR}.  We will prove the following theorem in \S\ref{subsec: comparison} showing an equational relation between the closed invariants $n_{\beta+h}$ and $N_{e-1,1}$:

 \begin{theorem}\label{theorem 1.2}
 $(-1)^{e-1}\cdot(e-1)\cdot p_\ast[\mathscr{M}(X,\beta+h, s)]^{\text{vir}}=[\mathscr{M}(W(\text{log}D),\beta,s)]^{\text{vir}}$ where $\beta$ is an effective curve class in $W$ and $h$ is the fiber class of $p: X\rightarrow W$.
 \end{theorem}

Then, as a direct consequence of theorem \ref{theorem 1.2}, we have:
\begin{corollary}\label{corollary 1.3}
$(-1)^{e}\cdot(e-1)\cdot n_{\beta+h}=N_{e-1,1}$ where $\beta$ is an effective curve class in $W$ and $h$ is the fiber class.
\end{corollary}

Finally, we can apply theorem \ref{theorem 1.1} and corollary \ref{corollary 1.3} to the case where $W=\mathbb{P}^2$ and $D$ is an elliptic curve to give a full understanding of the mirror ring $R(\mathbb{P}^2,D)$. Note that in this case, $e=3d$ when $\beta=dH$ where $H$ is a hyperplane class in $\mathbb{P}^2$.

First of all, as an application of the recursion formula, we can abstractly describe an enumerative behavior of 2-pointed relative Gromov-Witten invariants of $(\mathbb{P}^2, D)$ with a point condition where $D$ is a smooth cubic curve. Let us briefly preview the results deduced for $(\mathbb{P}^2,D)$.


\begin{proposition}
Given any positive integer $d$, for $a+b=3d$, the invariants $N_{ab0}^d$ and $N_{a,b}$ are completely determined by the number $N_{3d-1,1}$ plus those lower degree invariants.
\end{proposition}

\medskip

Then to complete an understanding of $R(\mathbb{P}^2,D)$ in this case, we can directly apply theorem \ref{theorem 1.2} to have the following corollary 

\begin{corollary}
$(-1)^{3d}\cdot(3d-1)\cdot n_{\beta+h}=N_{3d-1,1}$ where $\beta=dH$ as a curve class in $\mathbb{P}^2$ and $h$ is the fiber class.
\end{corollary}

Moreover, using these tools, we computed out all the degree 2 punctured Gromov-Witten invariants and all degree 2, 2-pointed relative Gromov-Witten invariants with a point condition for $(\mathbb{P}^2, D)$ as follows
\begin{corollary}
We have $N_{1,5}=1, N_{5,1}=25, N_{2,4}=7/2, N_{4,2}=14$, $N_{3,3}=9, N_{240}^2=N_{420}^2=42, N_{150}^2=N_{510}^2=30$ and $N_{330}^2=54$.
\end{corollary}
\begin{remark}
The numbers $N_{ab0}^d$ are just the number of degree $d$ rational curves tangent to order $a,b$ at two unspecified points of $D$ respectively, passing through a specified point away from $D$, see the remark \ref{rem 3.1.1}.
\end{remark}
\begin{remark}
In\cite{GRZ}, T. Grafnitz, H. Ruddat and E. Zaslow also computed out various 2-point Gromov-Witten invariants for $(\mathbb{P}^2, D)$, even in higher genus by the tropical correspondence results for smooth del Pezzo Calabi-Yau pairs proven in \cite{Gra} and \cite{Gra1}. Roughly speaking, their result is the computation of the broken line expansion of theta functions $\theta_q$ for a toric del Pezzo surface with smooth $D$ in terms of tropical invariants, log invariants, and hence a explicit computation of the Landau-Ginzburg potential $\theta_1$.   
\end{remark}
\textbf{Acknowledgement.} At first, the author particularly wants to express my sincere gratitude to my Ph.D advisor, Mark Gross, who taught me everything I've used in this paper. The author also wants to thank Helge Ruddat for his support and several useful conversations about degeneration formula and calculations of open Gromov-Witten invariants. The author also benefits from conversations with F. You and T. Grafnitz. The Author is funded by the ERC grant MSAG awarded by the 
European Research Council.

\medskip

\textbf{Notations:} Throughout the whole paper, we work over an algebraically closed field $\Bbbk$ of characteristic 0, and we assume that all relevant logarithmic structures are Zariski. For a logarithmic scheme $X$ and a map $f:X\longrightarrow Y$ between logarithmic schemes $X$ and $Y$, $\underline{X}$ and $\underline{f}$ represent the underlying scheme of $X$ and the underlying map of $f$ respectively. For a point $\text{Spec}(\Bbbk)$, $(\text{Spec}(\Bbbk),Q)$ means the logarithmic point with the logarithmic structure whose the ghost sheaf is $Q$. \\
For any monoid $Q$, set $Q^*:=\text{Hom}_{\text{Mon}}(Q,\mathbb{Z}), Q^{\vee}:=\text{Hom}_{\text{Mon}}(Q,\mathbb{N})$ and $Q^{\vee}_{\mathbb{R}}:=\text{Hom}_{\text{Mon}}(Q,\mathbb{R}_{\geq0})$, here \textbf{Mon} represents the category of monoids.

\section{Preliminaries and review}
\subsection{Basics on log/Relative GW invariants}
\ \ \ In this subsection, we will quickly review the basics about logarithmic Gromov-Witten theory introduced by Abramovich-Chen and Gross-Siebert in \cite{GS1}, \cite{Ch} and \cite{AC}. The results from the paper \cite{AMW} show that relative Gromov-Witten invariants constructed by using the method of so called expanded pair due to Jun Li are equivalent to logarithmic Gromov-Witten invariants in terms of curve counting. Therefore, we will use the terminologies logarithmic and relative interchangeably but we are not going to talk about relative Gromov-Witten theory in details, and readers can refer to the paper \cite{Li1} for detailed theory about relative invariants. For foundations of logarithmic geometry, reader can refer to K.Kato's paper \cite{Kk} and \cite{Og}.

A logarithmic version of the theory of stable curves has been studied by F.~Kato in \cite{Kf}, which turns out to be very powerful to investigate the smoothing property of a nodal curve around its nodes by putting an appropriate logarithmic structure. Thus, let us start off by recapping the basic definitions of log structure and log curves.
\begin{definition}
Let $X$ be a scheme. A \textit{logarithmic structure on $X$} is a sheaf of monoids $\mathcal{M}_X$ together with a morphism of sheaves of monoids $\alpha: \mathcal{M}_X\rightarrow \mathcal{O}_X $ such that the induced morphism $\alpha^{-1}(\mathcal{O}_X^*)\rightarrow \mathcal{O}_X^*$ is an isomorphism.
\end{definition}
\begin{remark}
A \textit{logarithmic scheme} is pair of a scheme $X$ with a logarithmic structure $\mathcal{M}_X$ on $X$. We simply call it log scheme later in this paper.
\end{remark}
\begin{definition}
The \textit{ghost sheaf} or \textit{characteristic sheaf} of a logarithmic structure is the quotient sheaf $\overline{\mathcal{M}}_X:=\mathcal{M}_X/\alpha^{-1}(\mathcal{O}_X^*)$.
\end{definition}
In our paper, all logarithmic structures are assumed to be fine unless otherwise stated. For the precise definition of fine or saturated logarithmic structure, reader can refer to \cite{Kk}. A fine and saturated log scheme is also simply called a fs log scheme. 
\ \ \begin{definition} 
A \textit{logarithmic curve} is a logarithmically smooth and flat morphism of fs log schemes $\pi: X\rightarrow S$ such that all geometric fibers are reduced and connected schemes of pure dimension 1 satisfying the following. 
If $\underline{U}\subset \underline{C}$ is the non-singular locus of $\underline{\pi}$, then there exist sections $x_1,\dotso, x_n$ of $\pi$ such that
$$\overline{\mathcal{M}}_C|_{\underline{U}}\cong \pi^*\overline{\mathcal{M}}_S\oplus\bigoplus_{i=1}^n(x_i)_*\mathbb{N}.$$
\end{definition}

Hence, the ghost sheaf $\overline{\mathcal{M}}_C$ has three
different possibilities shown in the following theorem:

\begin{theorem}\label{thm 2.1.1}
Assume that $\pi: C\rightarrow S$ is a log curve. Then
\begin{enumerate}
    \item fibers have at worst nodes as singularities.
    \item \'{e}tale locally on $S$, we can choose disjoint sections $x_i: S\rightarrow C$ in the smooth locus of $\pi$ whose images are called marked points such that: 
    \begin{enumerate}
        \item If $\eta$ is a general point in $C$ away from the marked points and nodes, then $$\overline{\mathcal{M}}_{C,\eta}\cong \overline{\mathcal{M}}_{S,\pi(\eta)}.$$
        \item If $p$ is a marked point in $C$, then $$\overline{\mathcal{M}}_{C,p}\cong \overline{\mathcal{M}}_{S,\pi(p)}\oplus \mathbb{N}.$$
        \item If $q$ is a node of $\pi^{-1}\pi(q)$ and $Q:=\overline{\mathcal{M}}_{S,\pi(q)}$, then
        $$\overline{\mathcal{M}}_{C,q}\cong Q\oplus_{\mathbb{N}}\mathbb{N}^2.$$
        where the map $\mathbb{N}\rightarrow \mathbb{N}^2$ is the diagonal map and the map $\mathbb{N}\rightarrow Q$ given by $1\mapsto \rho$ is some homomorphism of monoids uniquely determined by the map $\pi$ with $\rho\neq 0$
    \end{enumerate}
\end{enumerate}
\end{theorem}
Roughly speaking, based on the theorem above, the benefit of applying log geometry to the theory of moduli space of stable curves is that log smoothness often allows mild singularities (e.g. nodes) to occur. In a nutshell, log geometry techniques sometimes magically put us back in category of smooth spaces when we deal with something with mild singularities. 

\medskip

With the preparation above, we are finally ready for the definition of stable log maps.
\begin{definition}[\cite{GS1}]
Let $g:X\rightarrow W$ be a morphism of log schemes. A \textit{pre-stable log map with $n$ markings} to $X$ is a commutative diagram of morphisms of log schemes
\begin{center}
\begin{tikzcd}
C \arrow[r, "f"]\arrow[d, "\pi"]& X\arrow[d, "g"]\\
S \arrow[r, "f_S"] & W
\end{tikzcd}
\end{center}
where $\pi: C\rightarrow S$ is a log curve with $n$ mutually disjoint sections ${x_1, \dots, x_n}$ such that the image of each $x_i$ lies in the smooth locus of $\pi$.\\
Furthermore, a pre-stable log map is called \textit{stable} if the underlying pre-stable map of schemes forgetting the log structures is stable in the ordinary sense. 
\end{definition}
In general, the moduli space of stable log maps to a log scheme $X$ will not be of finite type. Hence we have to get rid of some less important log maps to make moduli space of finite type. The notion of \textit{basic} stable log maps due to Abramovich-Chen and Gross-Siebert in \cite{Ch},\cite{AC} and \cite{GS1}
now appear, and  
the idea is to only keep those kind of stable log maps which become ``universal" tropical map after tropicalization. For more details, the reader can refer to section 1 of \cite{GS1}.

Once we impose the basicness condition on stable log maps, we obtain 
a moduli space. Then we have the following theorem due to Gross and Siebert.
\begin{theorem}[Proposition 5.1, \cite{GS1}]
If $g:X\rightarrow W$ is log smooth, then the moduli space of basic stable log maps to $X$ with fixed contact orders and curve class is a Deligne-Mumford stack and carries a relative perfect obstruction theory relative to moduli stack of pre-stable logarithmic curves, and therefore possesses a virtual fundamental class. Moreover, if the map $g$ is proper, then the moduli stack is proper.
\end{theorem}

\subsection{Stable punctured log maps }

\begin{definition}[\cite{ACGS2}]\label{2.5}
Let $Y=(\underline{Y}, \mathcal{M}_Y)$ be a fine and saturated logarithmic scheme with a decomposition $\mathcal{M}_Y=\mathcal{M}\oplus_{\mathcal{O}^{\times}}\mathcal{P}$. A \textit{puncturing} of $Y$ along $\mathcal{P}\subset \mathcal{M}_Y$ is a fine sub-sheaf of monoids $$\mathcal{M}_{Y^{\circ}}\subset \mathcal{M}\oplus_{\mathcal{O^{\times}}}\mathcal{P}^{\text{gp}}$$
containing $\mathcal{M}_Y$ with a structure map $\alpha_{Y^{\circ}}: \mathcal{M}_{Y^{\circ}}\rightarrow \mathcal{O}_Y$ such that 
\begin{enumerate}
    \item The inclusion $p^\flat: \mathcal{M}_Y\rightarrow \mathcal{M}_{Y^\circ}$ is a morphism of logarithmic structures on $\underline{Y}$.
    \item For any geometric point $\overline{x}$ of $\underline{Y}$, let $s_{\bar{x}}\in \mathcal{M}_{Y^\circ, \overline{x}}$ be such that $s_{\overline{x}}\notin \mathcal{M}_{\overline{x}}\oplus_{\mathcal{O}^\times}\mathcal{P}_{\overline{x}}$. Representing $s_{\overline{x}}=(m_{\overline{x}}, p_{\overline{x}})\in \mathcal{M}_{\overline{x}}\oplus_{\mathcal{O}^\times}\mathcal{P}_{\overline{x}}$, we have $\alpha_{Y^\circ}(s_{\overline{x}})=\alpha_{\mathcal{M}_Y}(m_{\overline{x}})=0$ in $\mathcal{O}_{Y, \overline{x}}$. 
\end{enumerate}
We call a puncturing $\mathcal{M}_{Y^\circ}$ \textit{trivial} if the induced map $p^\flat$ is an isomorphism. Denote by $Y^\circ=(\underline{Y}, \mathcal{M}_{Y^\circ})$.
\end{definition}
\begin{remark}
Note that unlike stable logarithmic curves/maps, the logarithmic structure put on a punctured log curve is \textbf{not} necessarily saturated, in other words, $C^\circ$ is in general only a fine logarithmic scheme. 
\end{remark}
\begin{remark}
Readers can easily see that a puncturing of a log structure is not unique. Nonetheless, once a log scheme with a choice of puncturing is equipped with a log morphism to another log scheme, there is in fact a smallest choice for puncturing. More precisely, we have the following definiton.
\end{remark}
\begin{definition}\label{2.6}
Let $X$ be a log scheme. A morphism $f: Y^\circ\rightarrow X$ from a punctured log scheme is said to be \textit{pre-stable} if the puncturing $\mathcal{M}_{Y^\circ}$ is generated as a sheaf of fine monoids by $\mathcal{M}_Y$ and $f^\flat(f^* \mathcal{M}_X)$.
\end{definition}
Throughout the paper, we will essentially only be interested in the case that $Y^\circ$ is a punctured log scheme with the underlying log scheme a logarithmic curve $Y$ over a fine and saturated log scheme $S$. 

Note that when we are given a log curve $\pi:C\rightarrow S$, then by Theorem \ref{thm 2.1.1}, we have $\mathcal{M}_C=\mathcal{M}\oplus_{\mathcal{O}^\times}\mathcal{P}$ where $\mathcal{M}$ is the the log structure on $C$ with no
marked points and $\mathcal{P}$ is the logarithmic structure associated to the marked points. Therefore, it yields the definition of \textit{punctured logarithmic curve}.
\begin{definition}
A \textit{punctured logarithmic curve} parametrized by a fine and saturated log scheme $S$ is the following data:
$$(C^\circ\xrightarrow{p} C\xrightarrow{\pi} S, \textbf{x}=(x_1, x_2,\dotso, x_n))$$
where
\begin{enumerate}
    \item $(C\xrightarrow{\pi}S, \textbf{x}=(x_1, x_2, \dotso, x_n))$ is a logarithmic curve with $n$ disjoint marked points $x_1,\dotso, x_n$.
    \item $\mathcal{M}_{C^\circ}$ is a choice of a puncturing of $\mathcal{M}$ along the log structure $\mathcal{P}$ associated to the marked points.
\end{enumerate}
Furthermore, given a punctured logarithmic curve as defined above, a pre-stable \textit{punctured logarithmic map} is a diagram
\begin{center}
\begin{tikzcd}
C^\circ\arrow[d, "p"]\arrow[r, "f"] & X\\
C\arrow[d,"\pi" ]\\
S
\end{tikzcd}
\end{center}
where $f$ is pre-stable in the sense of Definition \ref{2.6}, and the pre-stable punctured logarithmic map is called \textit{stable} if forgetting about all logarithmic structures, the diagram above is a stable map in the ordinary sense. 
\end{definition}
\begin{remark}
The notion of puncturing of a stable log map along the marked points allows us to talk about negative contact orders. More concretely, suppose given a stable punctured log map $f:C^\circ\rightarrow X$ and $x\in \underline{C}^\circ$  a marked point. Then we have a chain of maps of monoids, denoting the composition by $u_x$:
$$P_x:= \mathcal{M}_{X, f(x)}\rightarrow \mathcal{M}_{C^\circ, x}\hookrightarrow \mathcal{M}_{S, \pi(x)}\oplus_{\mathcal{O}^\times}\mathbb{N}^{\text{gp}}\xrightarrow{pr_2} \mathbb{Z}$$
where $pr_2$ is the second projection map and $u_x\in P_x^*$ is the \textit{contact order} of $f$ at the point $x$. 

Notice that $x$ is sometimes called a \textit{marked point} if $u_x\in P_x^\vee$, otherwise, it is called a \textit{punctured point}.
\end{remark}
\begin{remark}
In general, imposing well-defined contact orders at punctured points is a quite subtle thing. For a full discussion, we refer readers to \cite{ACGS2}. Roughly speaking, given a family of punctured logarithmic maps $f: C^\circ/W\rightarrow X$, at each geometric point $w\in W$ and for each punctured point $x\in C_w^\circ$, we have the contact order $u_x: P_x\rightarrow \mathbb{Z}$ defined as above, i.e. we specify an integral tangent vector $u_x$ to $\sigma_{f(x)}\in\Sigma(X)$ (see section \ref{subsection 2.3}). Then as $w$ varies on $W$, the cones $\sigma_{f(x)}$ might vary, hence we have to consider the notion so-called \textit{family of contact orders} and its \textit{connected components}.\\
However, in this paper, especially for the main gluing theorem, we will stick to the case where $X$ is a smooth projective variety with the divisorial logarithmic structure given by a smooth divisor. So, the tropicalization $\Sigma(X)$ is just $\mathbb{R}_{\geq0}$, and there is not any issue to impose contact orders at punctures. 
\end{remark}
\subsection{Tropical interpretations}\label{subsection 2.3}
\ \ \ Tropical geometry supplies an efficient way for us to capture the combinatorial data of stable log maps or stable punctured log maps. In most cases, grouping stable (punctured) log maps by means of their tropical data inside the corresponding moduli space will result in a good stratification of the moduli space. Moreover, the notion \textit{basicness} for stable (punctured) log maps can be extracted naturally by looking at the tropical picture. 

Recall that for any logarithmic scheme $X$, we can associate to it its tropicalization $\Sigma(X)$ functorially which is a generalized cone complex in general, see \cite{ACGS2}. Thus, when we have a stable log curve 
\begin{center}
\begin{tikzcd}
C\arrow[d, "\pi"]\arrow[r, "f"] & X\\
S
\end{tikzcd}
\end{center}
we have a corresponding tropical picture 
\begin{center}
\begin{tikzcd}
\Sigma(C)\arrow[d, "\Sigma(\pi)"]\arrow[r,"\Sigma(f)"] & \Sigma(X)\\
\Sigma(S)
\end{tikzcd}
\end{center}
Throughout the entire paper, we are mainly interested in the case that either $S$ is a log point or it is a stacky point with a log structure (e.g. $B\mathbb{G}_m^\dagger$). \\

Here we will focus on the first case, that is, if $S=(\text{Spec}(k), Q)$, then $\Sigma(S)=Q_{\mathbb{R}}^\vee$. Let us recap how we get the cone complex $\Sigma(C)$.

First of all, at the generic point $\eta$ of each irreducible component of $C$, by Theorem \ref{thm 2.1.1}, we have $\overline{\mathcal{M}}_{C,\eta}\cong Q$, and thus it produces a cone $Q_{\mathbb{R}}^\vee$.

Secondly, at each marked point $p$, again by Theorem \ref{thm 2.1.1}, $\overline{\mathcal{M}}_{C,p}\cong Q\oplus\mathbb{N}$. Then $\text{Hom}(\overline{\mathcal{M}}_{C,p}, \mathbb{R}_{\geq 0})\cong Q_{\mathbb{R}}^\vee\times\mathbb{R}_{\geq 0}$. So, each marked point offers a cone $Q_{\mathbb{R}}^\vee\times\mathbb{R}_{\geq 0}$.

Thirdly, at each node $q$, Theorem \ref{thm 2.1.1} tells us that $\overline{\mathcal{M}}_{C,q}\cong Q\oplus_{\mathbb{N}}\mathbb{N}^2$ where the map $\mathbb{N}\rightarrow Q$ maps $1$ to some non-zero element $\rho$ and the map $\mathbb{N}\rightarrow\mathbb{N}^2$ maps $1$ to $(1,1)$. Then we have $\text{Hom}(\overline{\mathcal{M}}_{C,q}, \mathbb{R}_{\geq 0})\cong Q_{\mathbb{R}}^\vee\times_{\mathbb{R}_{\geq 0}}\mathbb{R}_{\geq 0}^2$ where the map $Q_{\mathbb{R}}^\vee\rightarrow\mathbb{R}_{\geq 0}$ is given by evaluation at $\rho$ and the map $\mathbb{R}_{\geq 0}^2\rightarrow\mathbb{R}_{\geq 0}$ maps $(a,b)$ to $a+b$. Thus, we can easily notice that we have an isomorphism $Q_{\mathbb{R}}^\vee\times_{\mathbb{R}_{\geq 0}}\mathbb{R}_{\geq 0}^2\cong \{(m, \lambda)\in Q_{\mathbb{R}}^\vee\times\mathbb{R}_{\geq 0}|\lambda\leq m(\rho)\}$.

Finally, $\Sigma(C)$ is obtained by gluing all possible cones described above using natural corresponding generization maps. More details are shown in \cite{ACGS1}.

In the meanwhile, it is also shown in \cite{ACGS1}, Prop. 2.25, that the map $\Sigma(\pi)$ together with data described above actually gives rise to a family of abstract tropical curves over $Q^{\vee}_{\mathbb{R}}$ written as a triple $\Gamma=(G, \textbf{g}, l)$ where $G$ is the dual intersection graph of $C$ with sets $V(G), E(G), L(G)$ of vertices, edges and legs, and the maps
$$\textbf{g}: V(G)\rightarrow \mathbb{N}, \text{and}\ l: E(G)\rightarrow \text{Hom}(Q^*, \mathbb{N})\setminus {0}.$$

Conversely, given such a triple $\Gamma=(G, \textbf{g}, l)$ of a family of tropical curves, we are able to construct a generalized cone complex, denoted by 
$\Gamma(G,l)$. This has one cone $\omega_x$ for each $x\in V(G)\cup E(G)\cup
L(G)$, with $\omega_v=Q_{\mathbb{R}}^\vee$ for $v \in V(G)$.

There is a more or less parallel tropical interpretation for punctured logarithmic maps in which the only change is that a leg may have finite length, in other words, we only need to do a slight modification for punctured points. Indeed, suppose $L\in L(G)$ corresponds to a punctured point (not marked point) with a puncturing $Q^\circ\subset  Q\oplus\mathbb{Z}$ that contains $Q\oplus\mathbb{N}$ as its proper submonoid. Set  $\omega_L=\text{Hom}(Q^\circ, \mathbb{R}_{\geq 0})$. By the finite generation property of $Q$, one can quickly show that there is a function $l(L):Q_{\mathbb{R}}^*\rightarrow \mathbb{R}_{\geq 0}$ such that 
$$\omega_L=\{(s,\lambda)\in Q_{\mathbb{R}}^*\times\mathbb{R}_{\geq 0}| \lambda\leq l(L)(s)\}.$$
Therefore, this tiny difference motivates the following definition of \textit{punctured tropical curves} over a monoid $Q$.

\begin{definition}
A family of \textit{punctured tropical curves} over a monoid $Q$ is a graph $G$ together with two maps
$$\textbf{g}: V(G)\rightarrow \mathbb{N},\ \  \text{and}\ l:E(G)\rightarrow \text{Hom}(Q, \mathbb{R}_{\geq 0}).$$

Furthermore, a family of \textit{punctured tropical maps} over a monoid $Q$ is a map of cone complexes $h:\Gamma\rightarrow\Sigma(X)$ for $X$ a logarithmic scheme and $\Gamma$ associated to the family of tropical curves $(G, \textbf{g}, l)$.
\end{definition}

From this definition, we can directly see that tropicalization of punctured log maps yields a family of punctured tropical maps to $\Sigma(X)$. More importantly, the corresponding punctured tropical maps of a punctured log map contains 
a collection of combinatorial information encoded by log structures. 

In summary, we can extract the following combinatorial data from a punctured log maps out of its tropicalization:
\begin{enumerate}
    \item A family of punctured tropical curves $\Gamma=(G,\textbf{g}, l)$.
    \item A map $\sigma: V(G)\cup E(G)\cup L(G)\rightarrow\Sigma(X)$ which associates to each object of $G$ the minimal cone of $\Sigma(X)$ that the object gets mapped to.
    \item For each edge $E_q$ corresponding to a node $q$, we have a \textit{weight vector} $u_q\in N_{\sigma(E_q)}$(the lattice of integral vectors to the cone $\sigma(E_q)$).
    \item For each $E_p$ corresponding to a marked point or punctured point, we have a contact order $u_L\in N_{\sigma(L)}$.
    \item A  map of cone complexes $h: \Gamma(G,l)\rightarrow \Sigma(X)$ 
such that if $L\in L(G)$ is a leg with a vertex $v$, then $h(\text{Int}(\omega_L))\subset\text{Int}(\sigma(L))$ and 
    $$h(\omega_L)=(h(\omega_v)+\mathbb{R}_{\geq 0}u_L)\cap\sigma(L)\subset N_{\sigma(L)}\otimes_\mathbb{Z}\mathbb{R}.$$
If $h(\omega_v)+\mathbb{R}_{\ge 0}u_L\subseteq \sigma(L)$, then we call
the leg $L$ a marked leg; otherwise it is a punctured leg.
    \item if $v_1, v_2$ are vertices of an edge $E_q$ from $v_1$ to $v_2$, then $h(\text{Int}(E_q))\subset\text{Int}(\sigma(E_q))$, and satisfy the equation
    $$h(v_2)-h(v_1)=l(E_q)u_q.$$
\end{enumerate}

Sometimes, a collection of data described as above for a family of (punctured) logarithmic maps is called \textit{combinatorial type} of the (punctured) logarithmic maps, denoted by $$\tau=(G,\pmb{g},\sigma,\pmb{u})$$
in which $\pmb{u}$ is the package $\{u_p, u_q\}$ collecting all the information of contact orders at punctures or nodes, and define $\pmb{\tau}:=(\tau, \beta)$ to be the combinatorial type $\tau$ with curve class $\beta$.\\

Now we are in position to introduce a very crucial notion in punctured Gromov-Witten theory, which is the notion of \textit{basicness} of a punctured log map.

\begin{definition}
A pre-stable punctured logarithmic map $(C/S, \textbf{p}, f:C\rightarrow X)$ over a log point $S=(\text{Spec}(\Bbbk), Q)$ is said to be \textit{basic} if the associated family of tropical maps $$h:\Gamma(G,l)\rightarrow\Sigma(X)$$ over $Q^{\vee}_{\mathbb{R}}$ is universal among all tropical maps of the same combinatorial type.
\end{definition}
\begin{remark}\label{rem 2.3.1}
We can also restrict $\Sigma(f)$ to fibers of $\Sigma(\pi)$ to get maps between cone complexs
$$\Sigma_m:=\Sigma(f)|_{\Sigma(\pi)^{-1}(m)}: \Sigma(\pi)^{-1}(m)\rightarrow\Sigma(X),$$
for each $m\in\Sigma(S)$.
\end{remark}
\subsection{Artin fans}
\ \ \ In general, an \textit{Artin fan} is a logarithmic Artin stack that is logarithmically \'{e}tale over a point Spec($k$). In logarithmic algebraic geometry, there is a way of associating to any logarithmic scheme $X$ a canonical Artin fan $\mathcal{X}$ such that there is an initial factorization $X\rightarrow\mathcal{X}\rightarrow\textbf{Log}_k$ for the canonical map $X\rightarrow\textbf{Log}_k$ introduced by M. Olsson in \cite{Ol1}. Here $\textbf{Log}_k$ is the Olsson's stack parametrizing all fine logarithmic structures. The detailed construction of $\mathcal{X}$ can be found in \cite{ACMW}. \\

The appearance of $\mathcal{X}$ can rephrase many properties about $X$ in logarithmic geometry. For instance, $X$ is logarithmically smooth if and only if the associated map $X\rightarrow\mathcal{X}$ is smooth. Furthermore, the following proposition in \cite{ACGS1} reflects an importance of Artin fans in the sense of encoding combinatorial data of log schemes given by their tropicalizations.

\begin{proposition}[\cite{ACGS1},Prop.\ 2.10]
Let $X$ be a Zariski fs log scheme log smooth 
over ${\operatorname{Spec}}(\Bbbk)$. Then for any fs log scheme $T$, there is a canonical bijection 
$${\operatorname{Hom}}_{\operatorname{fs}}(T, \mathcal{X})\rightarrow
{\operatorname{Hom}}_{\operatorname{Cones}}(\Sigma(T),\Sigma(X)),$$
which is functorial in $T$.
\end{proposition}

\subsection{Relevant moduli spaces}
\label{subsec:relevant}
\ \ \ Throughout this sub-section, we fix a log smooth morphism $X\rightarrow W$ where $X$ carries a Zariski logarithmic structure and $W=(\text{Spec}(\Bbbk), Q)$ is a log point. We further assume that the log structure $\overline{\mathcal{M}}_X$ is globally generated, that is, the natural map $\Gamma(X,\overline{\mathcal{M}}_X)\rightarrow \overline{\mathcal{M}}_{X,p}$ is surjective for every $p\in X$. In this section, we essentially follow the notions and terminologies used in \cite{GS2}, section 3.
\begin{remark}\label{Remark 3.0.1}
Let $r$ be an integral point in a cone of $\Sigma(X)$; we
write $r\in \Sigma(X)(\mathbb{Z})$. We may view
$-r$ as a contact order of a punctured log map to $X$ at a punctured point, and let $\sigma$ be the minimal cone in $\Sigma(X)$ containing $r$. Then $\sigma$ corresponds to a locally closed stratum $Z^\circ$ in $X$. Let $Z$ be the closure of $Z^\circ$ in $X$. Thus, $r$ can be view as an element of $\text{Hom}(\overline{\mathcal{M}}_Z, \mathbb{N})$ given as follows. 
Let $\eta$ be the generic point of $Z$, so we have $\sigma=\sigma_{\eta}=\text{Hom}(\overline{\mathcal{M}}_{Z,\eta},\mathbb{R}_{\geq 0})$. Hence, for any section $s\in\overline{\mathcal{M}}_Z$, we obtain its germ $s_{\eta}\in\overline{\mathcal{M}}_{Z,\eta}$. So we just define $r(s):=r(s_{\eta})$.
\end{remark}

\ \ \ In this subsection, we are going to introduce several related moduli spaces, and then recall the procedure of imposing a point-constraint at a punctured point on moduli spaces. Here, we give the basic setup for our target space $X$.

\subsubsection{Moduli of punctured maps} 
\ \ \ First of all, we fix a type $\pmb{\beta}$ of punctured map, that is, a curve class $\beta\in \text{H}_2(X,\mathbb{Z})$, a genus 
assigned to each irreducible component, a number of punctured points with a choice of contact order for each of punctured points. Since we are going to only consider genus 0 case, we fix genus to be 0 for source curves once and for all.\\

Therefore, we have a moduli space of stable basic punctured log maps to $X$ denoted by $\mathscr{M}(X,\pmb{\beta})$, and then if we forget the curve class $\beta$ from the type $\pmb{\beta}$, then we get a moduli space of pre-stable punctured log maps to $\mathcal{X}$ denoted by $\mathfrak{M}(\mathcal{X},\pmb{\beta})$. Moreover, we have a natural map 
$$\mathscr{M}(X,\pmb{\beta})\rightarrow\mathfrak{M}(\mathcal{X},\pmb{\beta})$$
via the composition of stable punctured log maps with the canonical map $X\rightarrow\mathcal{X}$. Based on the theory of punctured log maps developed in \cite{ACGS2}, the map above possesses a perfect relative obstruction theory and hence a virtual pullback of cycles via \cite{Ma}. However, in general, $\mathfrak{M}(\mathcal{X},\pmb{\beta})$ might behave very badly, e.g. it is not equi-dimensional, so $\mathscr{M}(X,\pmb{\beta})$ might not possess a virtual fundamental class. \\

More specifically, we can talk about punctured logarithmic maps marked by a combinatorial type $\tau$ or $\pmb{\tau}$, and consider the corresponding moduli stacks $\mathscr{M}(X,\tau)$ or $\mathscr{M}(X,\pmb{\tau})$. However, defining $\mathscr{M}(X,\tau)$ or $\mathscr{M}(X,\pmb{\tau})$ is a bit subtler than that of $\mathscr{M}(X,\pmb{\beta})$, and we refer the reader to \S3 of \cite{ACGS2} for the precise definitions of $\mathscr{M}(X,\tau)$ and $\mathscr{M}(X,\pmb{\tau})$.

\subsubsection{Evaluation space}
\ \ \ Next, we are going to construct a moduli space parametrizing punctures in $X$ with a given contact order $-r$ where $r\in\Sigma(X)(\mathbb{Z})$ and its universal family of such punctures. We will follow \cite{GS2} for the notations of the space parametrizing punctures with contact order $-r$ and its universal family, which are denoted by $\mathscr{P}(X,r)$ and $\widetilde{\mathscr{P}}(X,r)$ respectively.\\

Let us recall the constructions of $\mathscr{P}(X,r)$ and $\widetilde{\mathscr{P}}(X,r)$ from \cite{GS2}. First, let $Z:=Z_r$ be the closed stratum
indexed by $r\in \Sigma(X)(\mathbb{Z})$ as in Remark \ref{Remark 3.0.1}.
Then we set 
$$\widetilde{\mathscr{P}}(X,r):=Z\times B\mathbb{G}_m^\dagger$$
where $Z$ inherits an induced log structure from $X$ and $B\mathbb{G}_m^\dagger$ is the classifying stack $B\mathbb{G}_m$ equipped with the log structure pulled back from the divisorial log structure on $[\mathbb{A}^1/\mathbb{G}_m]$ with respect to $B\mathbb{G}_m$.\\

Next, we define the moduli space $\mathscr{P}(X,r)$ parametrizing punctures with contact order $-r$. We define $\mathscr{P}(X,r)$ to have the same underlying stack as $\widetilde{\mathscr{P}}(X,r)$ with the log structure defined as follows. Firstly, we define $\overline{\mathcal{M}}_{\mathscr{P}(X,r)}$ as the sub-sheaf of $\overline{\mathcal{M}}_{\widetilde{\mathscr{P}}(X,r)}=\overline{\mathcal{M}}_Z\oplus\mathbb{N}$ given by 
$$\overline{\mathcal{M}}_{\mathscr{P}(X,r)}(U):=\{(m,r(m))\ |\ m\in\overline{\mathcal{M}}_Z(U)\}$$
where $r$ can be viewed as an element of $\text{Hom}(\overline{\mathcal{M}}_Z, \mathbb{N})$ according to Remark \ref{Remark 3.0.1}.\\
Now we may define a logarithmic structure on $\mathscr{P}(X,r)$ as 
$$\mathcal{M}_{\mathscr{P}(X,r)}:=\overline{\mathcal{M}}_{\mathscr{P}(X,r)}\times_{\overline{\mathcal{M}}_{\widetilde{\mathscr{P}}(X,r)}}\mathcal{M}_{\widetilde{\mathscr{P}}(X,r)}.$$
Analogously, we can define $\mathscr{P}(\mathcal{X},r)$. The whole story about moduli space of punctures can also be developed in parallel for the associated Artin stack $\mathcal{X}$ of $X$. We just substitute $\mathcal{X}$ for $X$ everywhere and keep everything else unchanged.\\
The next proposition reflects the universal property of $\mathscr{P}(X,r)$.
\begin{proposition}[\cite{GS2}, Proposition 3.3]
Let $f:C^\circ/W\rightarrow X$ be a pre-stable punctured log map with a punctured point $x:\underline{W}\rightarrow\underline{C}$ with contact order $-r$. Then there exists a canonical morphism $\text{ev}:W\rightarrow\mathscr{P}(X,r)$ with the property that 
$$W^\circ:=W\times^{\text{fine}}_{\mathscr{P}(X,r)}\widetilde{\mathscr{P}}(X,r)$$
agrees with $(\underline{W}, x^\ast\mathcal{M}_{C^\circ})$.\\
The analogous statements for punctured log maps $f: C^\circ/W\rightarrow\mathcal{X}$ also hold.
\end{proposition}
\begin{remark}\label{rem 2.5.2}
In fact, in \S3 of \cite{GS2}, they give $\mathscr{P}(\mathcal{X},r)$ the structure of an idealized logarithmic stack as follows. We can view $r$ as an element of $\text{Hom}(\overline{\mathcal{M}}_{\mathscr{P}(\mathcal{X},r)},\mathbb{N})$. Then let $\mathcal{I}\subset\overline{\mathcal{M}}_{\mathscr{P}(\mathcal{X},r)}$ be the ideal sheaf given by $r^{-1}(\mathbb{Z}_{>0})$. The we can lift it to an ideal sheaf $\mathcal{I}\subset\mathcal{M}_{\mathscr{P}(\mathcal{X},r)}$. This turns out to be a coherent idealized log structure on $\mathscr{P}(\mathcal{X},r)$. For details about idealized log structure, see \S III.1.3 of \cite{Og}.
\end{remark}
We fix a type $\pmb{\beta}$ of punctured log maps, and further assume that there is one punctured point $x_{\text{out}}$ with the contact order $-r$ where $r\in\Sigma(X)(\mathbb{Z})$. According to the proposition above, we have two maps of stacks 
$$\text{ev}_X:\mathscr{M}(X,\pmb{\beta})\rightarrow\mathscr{P}(X,r), \text{and}\ \text{ev}_{\mathcal{X}}:\mathfrak{M}(\mathcal{X},\pmb{\beta})\rightarrow \mathscr{P}(\mathcal{X},r).$$
Then we define a moduli space as follows.
\begin{definition}
We define 
$$\mathfrak{M}^{\text{ev}}(\mathcal{X},\pmb{\beta}):=\mathfrak{M}(\mathcal{X},\pmb{\beta})\times_{\underline{\mathcal{X}}}\underline{X}$$
where the map $\mathfrak{M}(\mathcal{X},\pmb{\beta})\rightarrow\underline{\mathcal{X}}$ is the evaluation map at $x_{\text{out}}$.
\end{definition}
\begin{remark}
Note that there is a factorization 
$$\mathscr{M}(X,\pmb{\beta})\xrightarrow{\epsilon}\mathfrak{M}^{\text{ev}}(\mathcal{X},\pmb{\beta})\rightarrow\mathfrak{M}(\mathcal{X},\pmb{\beta}) $$
and we have a relative perfect obstruction theory for $\epsilon$ which is described in detail in \cite{ACGS2},\S4.
\end{remark}
\begin{remark}
In general, notice that, for any subset $\pmb{S}$ of \{nodal sections, marked sections, punctured sections\}, we have a evaluation map $\mathfrak{M}(\mathcal{X},\pmb{\beta})\rightarrow\prod_{S\in\pmb{S}}\underline{\mathcal{X}}$ given by 
$$(C^\circ/W,\pmb{p},f)\mapsto (\underline{f}\circ p_S)_{S\in\pmb{S}},$$
Therefore, we can define the corresponding moduli space  
$\mathfrak{M}^{\text{ev},\pmb{S}}(\mathcal{X},\pmb{\beta})$ associated to $\pmb{S}$ as the following:
$$\mathfrak{M}^{\text{ev},\pmb{S}}(\mathcal{X},\pmb{\beta}):=\mathfrak{M}(\mathcal{X},\pmb{\beta})\times_{\prod_{S\in\pmb{S}}\mathcal{\underline{X}}}\prod_{S\in\pmb{S}} \underline{X}.$$
If there is no danger of confusion, we just leave out the superscript $\pmb{S}$ and write it simply as $\mathfrak{M}^{\text{ev}}(\mathcal{X},\pmb{\beta})$.
\end{remark}
\subsubsection{Point-constrained moduli spaces}
We start off by stating a proposition.
\begin{proposition}[\cite{GS2}, Proposition 3.8]
Fix $r\in \Sigma(X)(\mathbb{Z})$ and a closed point $z\in Z^\circ$ where $Z^\circ$ is the corresponding locally closed stratum determined by $r$. Then there is a morphism 
$$B\mathbb{G}_m^\dagger\rightarrow\mathscr{P}(X,r)$$
with the image $z\times B\mathbb{G}_m^\dagger$ which, on the level of the ghost sheaves, is given by 
$$r:\overline{\mathcal{M}}_{Z,z}\cong\overline{\mathcal{M}}_{\mathscr{P}(X,r),z}\rightarrow\mathbb{N}.$$
\end{proposition}
Now we are in position to define the so-called point-constrained moduli space.
\begin{definition}[\cite{GS2}, Definition 3.9]
We define 
\begin{align*}
    \mathfrak{M}^{\text{ev}(x_\text{{out}})}(\mathcal{X},\pmb{\beta},z) :&=\mathfrak{M}^{\text{ev}(x_{\text{out}})}(\mathcal{X},\pmb{\beta})\times_{\mathscr{P}(X,r)}B\mathbb{G}_m^\dagger\\
    &=\mathfrak{M}(\mathcal{X},\pmb{\beta})\times_{\mathscr{P}(\mathcal{X},r)}B\mathbb{G}_m^\dagger
\end{align*}
where $\mathfrak{M}^{\text{ev}}(\mathcal{X},\pmb{\beta})\rightarrow\mathscr{P}(X,r)$ is the evaluation map at $x_{\text{out}}$ and the map $B\mathbb{G}_m^\dagger\rightarrow\mathscr{P}(X,r)$ is given by the proposition above.
\end{definition}
Similarly, we can define 
$$\mathscr{M}(X,\pmb{\beta},z):=\mathscr{M}(X,\pmb{\beta})\times_{\mathscr{P}(X,r)}B\mathbb{G}_m^\dagger.$$
Then, we can easily see that there is a cartesian diagram (in all categories)
\begin{center}
\begin{tikzcd}
\mathscr{M}(X,\pmb{\beta},z)\arrow[d,"\epsilon'"]\arrow[r]&\mathscr{M}(X,\pmb{\beta})\arrow[d,"\epsilon"]\\
\mathfrak{M}^{\text{ev}}(\mathcal{X},\pmb{\beta},z)\arrow[r]&\mathfrak{M}^{\text{ev}}(\mathcal{X},\pmb{\beta})
\end{tikzcd}
\end{center}
Therefore, the relative perfect obstruction theory for $\epsilon$ pulls back to a relative perfect obstruction theory for $\epsilon'$, so the virtual pullback is defined via \cite{Ma}. However, as we already mentioned before, $\mathfrak{M}^{\text{ev}}(\mathcal{X},\pmb{\beta},z)$ may be very bad and not equi-dimensional in general, thus $\mathscr{M}(X,\pmb{\beta},z)$ might not possess a virtual fundamental class. However, since our goal is to illustrate some enumerative properties of punctured invariants, the moduli spaces that we will consider will be all nice and equipped with virtual cycles. \\

\section{The main gluing theorem}
\subsection{Splitting punctured logarithmic maps}\label{section 3.1}

Let us explain the setup in which we will be for the rest of the paper. We will fix an arbitrary smooth log Calabi-Yau pair $(X,D)$ where $X$ is a smooth projective variety with a smooth divisor $D$ satisfying $K_X+D=0$. Choose a general point $z\in D$ such that there is not any rational curve going through $z$ inside $D$. We can achieve this since $D$, as a smooth variety, is a (weak) Calabi-Yau variety. However, if for any point on $D$, there existed a rational curve passing through that point, then it would imply that $D$ is uniruled by 1.3 Proposition of \S4 in \cite{Ko}. So, by 1.11 Corollary of \S4 in \cite{Ko}, we could conclude that $D$ has Kodaira dimension $-\infty$, which would tell us that $D$ is by no chance a Calabi-Yau variety. 

Fix a type of curve $\pmb{\beta}$ for $X$ where $\pmb{\beta}$ consists of non-zero curve class $\beta\in \text{H}_2(X,\mathbb{Z})$, $2$ distinct ordinary marked points $x_1, x_2$ with prescribed contact orders $p, q\in\Sigma(X)(\mathbb{Z})$ respectively, and exactly one punctured point $x_{\text{out}}$ with contact order $-r\neq0\in\Sigma(X)(\mathbb{Z})$. By the point-constraint imposing process described above, we can form the point-constrained moduli spaces $\mathscr{M}(X,\pmb{\beta},z)$ and $\mathfrak{M}^{\text{ev}}(\mathcal{X},\pmb{\beta},z)$. 
The paper \cite{GS2} shows that the space $\mathfrak{M}^{\text{ev}}(\mathcal{X},\pmb{\beta},z)$ is actually pure-dimensional of dimension $0$. Further, $\mathscr{M}(X,\pmb{\beta},z)$ possesses a virtual fundamental class defined by the virtual pullback of the fundamental class of $\mathfrak{M}^{\text{ev}}(\mathcal{X},\pmb{\beta},z)$. The relative virtual dimension is $\beta\cdot c_1(\Theta_{X/\Bbbk})=\beta\cdot(K_X+D)=0$, and so the virtual dimension of $\mathscr{M}(X,\pmb{\beta},z)$ is $0$. See Proposition 
3.12 in \cite{GS2} for more details. 
Then these facts naturally yield the following definition of punctured Gromov-Witten invariants relevant for the cosntruction.

\begin{definition}\label{def 3.1.1}
Let $p_1, p_2, r\in\Sigma(X)(\mathbb{Z})$, and let $\pmb{\beta}$ be a type of punctured curve with underlying curve class $\beta$ and three punctured points $x_1, x_2, x_{\text{out}}$ with contact order $p, q$ and $-r$ respectively. We define
$$N_{pqr}^{\pmb{\beta}}:=\int_{[\mathscr{M}(X,\pmb{\beta},z)]^{\text{vir}}}1.$$
\end{definition}
\begin{remark}\label{rem 3.1.1}
In fact, in the definition of punctured Gromov-Witten invariants, we could allow $r$ to be $0$. Then the punctured invariants $N_{pq0}^{\pmb{\beta}}$ are just standard 3-point relative Gromov-Witten invariants with one point-constraint away from $D$, having tangency order $p, q$ respectively with $D$. 
\end{remark}

Next, we are going to briefly recap splitting and gluing operations in the theory of punctured logarithmic maps. In \S5 of \cite{ACGS2}, they defined the notion \emph{splitting type} for nodal sections of punctured logarithmic maps. Roughly speaking, a nodal section of a family of punctured logarithmic maps is of splitting type if \'etale locally, after partially normalizing this nodal section, we will end up with a trivial two-fold unbranched cover of the original nodal section. Then any nodal section will allow us to split the type of punctured logarithmic map, and then result in a splitting map which breaks our moduli space into many pieces. More precisely, by Proposition 5.15 of \cite{ACGS2}, for a nodal section of splitting type, we have a Cartesian diagram 
\begin{center}
\begin{equation}
\label{eq:diagram}
    \begin{tikzcd}
\mathscr{M}(X,\tau)\arrow[d]\arrow[r, "\delta"] & \mathscr{M}(X,\tau_1)\times\mathscr{M}(X,\tau_2)\arrow[d]\\
\mathfrak{M}^{\text{ev}}(\mathcal{X},\tau)\arrow[r,"\delta^{\text{ev}}"] & \mathfrak{M}^{\text{ev}}(\mathcal{X},\tau_1)\times\mathfrak{M}^{\text{ev}}(\mathcal{X},\tau_2)
\end{tikzcd}
\end{equation}
\end{center}
with both maps being finite and representable. 

\begin{remark}
In the situation we are interested in, all nodal sections are automatically of splitting type since our source curve is always rational. Hence, both $\delta^{\text{ev}}$ and $\delta$ always exist in our setup. 
\end{remark}
\begin{remark}\label{rem 3.2.1}
Notice that after splitting a punctured logarithmic map of combinatorial type $\tau$ at a nodal section into punctured logarithmic maps of combinatorial type $\tau_1, \tau_2$ respectively, we actually get two extra punctured points $w\in\tau_1, w'\in\tau_2$ such that the contact orders at these two points are negative to each other.  
\end{remark}
Since we are interested in some enumerative properties of punctured invariants, we will always take the point-constraint into account. Thus, we need a version of splitting maps with point-constraint. The following lemma deals with such a slight modification.

\begin{lemma}\label{Lemma 3.1.1}
Let $\tau$ be a combinatorial type with two input legs representing ordinary marked points and one output leg representing a punctured point such that one of the chosen input legs is adjacent to the same vertex $v_{\mathrm{out}}$ 
as the output leg. 
Assume given an edge $E$ also adjacent to $v_{\mathrm{out}}$. Given
a punctured map of type $\tau$,
let $\pmb{S}$ be the 2-point set consisting of the nodal section corresponding
to $E$ and the punctured point $x_{\mathrm{out}}$. Let $\tau_1$ and $\tau_2$ be the resulting types of punctured curve after splitting $\pmb{\tau}$ at $E$,
with $x_{\text{out}}$ in the component corresponding to $\pmb{\tau}_2$. 
Then we have the following cartesian diagram
\begin{center}

\begin{tikzcd}
\mathfrak{M}^{\text{ev}}(\mathcal{X}, \tau,z)\arrow[d]\arrow[r, "\widetilde{\delta}^{\text{ev}}"]& \mathfrak{M}^{\text{ev}}(\mathcal{X},\tau_1)\times\mathfrak{M}^{\text{ev}}(\mathcal{X},\tau_2,z)\arrow[d]\\
\mathfrak{M}^{\text{ev}}(\mathcal{X},\tau)\arrow[r,"\delta^{\text{ev}}"] & \mathfrak{M}^{\text{ev}}(\mathcal{X},\tau_1)\times\mathfrak{M}^{\text{ev}}(\mathcal{X},\tau_2)
\end{tikzcd}

\end{center}
and an analogous statement holds for $\mathscr{M}(X,\tau,z)$. Note that these evaluation spaces are with respect to $\pmb{S}$.\\
Hence, finiteness and representability of $\delta^{\text{ev}}$ imply that the point-constraint splitting map $\widetilde{\delta}^{\text{ev}}$ is finite and representable.
\end{lemma}
\begin{proof}
The lemma follows directly from the definition of point-constrained moduli space and \eqref{eq:diagram}. Indeed, note that $\mathfrak{M}^{\text{ev}}(\mathcal{X},\tau_2,z)=\mathfrak{M}^{\text{ev}}(\mathcal{X},\tau_2)\times_{\mathscr{P}(\mathcal{X},r)}B\mathbb{G}_m^\dagger$. Then the conclusion follows directly by taking the fiber product.
\end{proof}
\subsection{Gluing punctured logarithmic maps}
A general framework of gluing arbitrary punctured logarithmic maps has been developed by Abramovich, Chen, Gross and Siebert in \cite{ACGS2}, and in this subsection, we are going to just apply it to our setup, in other words, we want to reverse the splitting process which has been explained in the diagram (3.1).\\ 

Consider base schemes $W_1, W_2$ with maps $W_1\rightarrow\mathfrak{M}^{\text{ev}}(\mathcal{X},\tau_1)$ and $W_2\rightarrow\mathfrak{M}^{\text{ev}}(\mathcal{X},\tau_2)$, i.e., consider families of punctured logarithmic maps $(C_1^\circ/W_1,\tau_1,f_1:C_1^\circ\rightarrow X),(C_2^\circ/W_2,\tau_2,f_2:C_2^\circ\rightarrow X)$ parametrized by $W_1$ and $W_2$ respectively, and two sections $\underline{w}: \underline{W}_1\rightarrow\underline{C}_1^{\circ}, \underline{w}': \underline{W}_2\rightarrow\underline{C}_2^{\circ}$.

Obviously, if we want to glue these two families, we need to be able to glue them schematically. Therefore, it is reasonable to assume that $\underline{f}_1\circ\underline{w}=\underline{f}_2\circ\underline{w'}$. The chief difficulty that differs from the ordinary gluing situation of ordinary stable maps is that we do not have evaluation maps in logarithmic category. So the method 
used in \cite{ACGS2} is to enlarge the logarithmic structures of $W_1$ and $W_2$ as follows:\\

Let $W_1^E=(\underline{W_1},w^*\mathcal{M}_{C_1^\circ})$ and $\widetilde{W}_1$ be the saturation of $W_1^E$ (see \cite{Og}, III Prop. 2.1.5), and similarly we can define $\widetilde{W}_2$ as well. Then, it is not so hard to see that we have the following evaluation map

$$\widetilde{W}_i\longrightarrow W_i^E\longrightarrow X$$
for $i=1,2$. Then we have the following gluing proposition 
\begin{proposition}\label{prop 3.2.1}
There exists a Cartesian diagram in the category of fs log stacks 
\begin{center}
    \begin{tikzcd}
    \widetilde{W}\arrow[d]\arrow[r] & \widetilde{W}_1\times\widetilde{W}_2\arrow[d]\\
    X\arrow[r, "\Delta"] & X\times X
    \end{tikzcd}
\end{center} 
such that there is a logarithmic scheme $W=(\widetilde{\underline{W}}, \mathcal{M}_W)$ with $\mathcal{M}_W\subset\mathcal{M}_{\widetilde{W}}$, equipped with morphisms $\psi_i:W\rightarrow W_i$, $(i=1,2)$ and a universal glued family $(\pi: C^\circ\rightarrow W, \pmb{\beta}, f:C^\circ\rightarrow X)$.
\end{proposition}
\begin{proof}
This is just a special case of theorem 5.13 in \cite{ACGS2}.
\end{proof}
\begin{remark}
In fact, there is no barrier at all to extend such a gluing method to the situation with point-constraint at $x_{\text{out}}$.
\end{remark}
Recall that we have a splitting map $\Tilde{\delta}^{\text{ev}}$ in lemma \ref{Lemma 3.1.1}. The next proposition will give us a factorization of $\Tilde{\delta}^{\text{ev}}$ which can be used to relate punctured invariants to 2-pointed relative/logarithmic Gromov-Witten invariants with one-point constraint.

\begin{proposition}\label{Prop 3.1.1}
In the situation of the lemma \ref{Lemma 3.1.1} above, further assume that $r\neq0$. We then have a diagram
\begin{center}
    
 \begin{tikzcd}
 \mathscr{M}(X,\tau,z)\arrow[d]\arrow[r, "\Tilde{\phi}"] & (\mathscr{M}(X,\tau_1)\times_{\underline{D}}z)\times\mathscr{M}(X,\tau_2,z)\arrow[d, "\epsilon"]\arrow[r] & \mathscr{M}(X,\tau_1)\times\mathscr{M}(X,\tau_2,z)\arrow[d]\\
 \mathfrak{M}^{\text{ev}}(\mathcal{X},\tau,z)\arrow[r, "\phi"] & (\mathfrak{M}^{\text{ev}}(\mathcal{X},\tau_1)\times_{\underline{D}}z)\times \mathfrak{M}^{\text{ev}}(\mathcal{X},\tau_2,z)\arrow[d]\arrow[r, "\Tilde{\Delta}"] & \mathfrak{M}^{\text{ev}}(\mathcal{X},\tau_1)\times \mathfrak{M}^{\text{ev}}(\mathcal{X},\tau_2,z)\arrow[d]\\
 & \underline{D}\arrow[r, "\Delta"] & \underline{D}\times \underline{D}
 \end{tikzcd}
 
\end{center}
with all squares Cartesian in all categories. Furthermore, $\phi$ is a finite surjective morphism. 
\end{proposition}
\begin{proof}
Firstly, note that the types of curve class $\tau_1$ and $\tau_2$ incorporate an ordinary marked point $p$ and a punctured point $q$ respectively coming from splitting the chosen node. Note the component which contains $x_{\text{out}}$ gets mapped into the divisor $D$ because the contact order at $x_{\text{out}}$
namely $-r$, is negative. But for a generic $z\in D$, there is no rational curve passing through $z$ in $D$ because $D$, by the definition, is a Calabi-Yau variety. Hence, we can conclude that the map restricting to the component will be just the constant map with the value $z$. Therefore, the map $\phi$ is just the splitting map from any punctured logarithmic map of type $\tau$ into two punctured logarithmic maps of type $\tau_1$ and $\tau_2$ respectively, each of which has been equipped with a point-constraint about $z$. Further, $\Tilde{\Delta}$ is just the map forgetting the point-constraint given on $\mathfrak{M}^{\text{ev}}(\mathcal{X}, \tau_1)$. Then checking that the square involved in $\Delta$ and $\Tilde{\Delta}$ becomes routine. Secondly, the fact that the top two diagrams are Cartesian follows directly from the lemma \ref{Lemma 3.1.1} and the diagram (3.1), and the finiteness of $\phi$ follows from the finiteness of $\delta^{\text{ev}}=\Tilde{\Delta}\circ\phi$.

Finally, we need to show that $\phi$ is surjective. In other words, we need to show that given $W_1\rightarrow\mathfrak{M}^{\text{ev}}(\mathcal{X},\tau_1)$ and $W_2\rightarrow\mathfrak{M}^{\text{ev}}(\mathcal{X},\tau_2,z)$, based on proposition \ref{prop 3.2.1}, the fine and saturated fiber product $\widetilde{W}_1\times_X^{fs}\widetilde{W}_2$ is not empty. 

Let us observe that the only possibility to make the fs fiber product empty is that the fine fiber product $\widetilde{W}_1\times_X^f\widetilde{W}_2$ is empty 
since going from the fine to the fs fibre product is saturation, which
is always surjective, see \cite{Og}, III,Prop.~2.1.5.
However, in this case, $\widetilde{W}_1\times_X^f\widetilde{W}_2$ is not empty since the morphisms $\widetilde{W}_i\rightarrow X$ are integral
because any non-zero morphism from the free rank 1 monoid is integral.
\end{proof}

\begin{remark}
Roughly speaking, the factorization diagram in the proposition above indicates that gluing punctured logarithmic maps splits into 2 steps, that is, gluing them schematically at first and then gluing the logarithmic structures. We remark that the map $\phi$ showing up in the diagram above needs not be surjective in general and the case in which the map $\phi$ is surjective is called \textit{tropical transverse}, see \cite{Gro} for a detailed discussion about tropical transversality, and calculating the degree of $\phi$ is literally the key point of relating punctured invariants to usual logarithmic invariants.
\end{remark}

\subsection{The main gluing formula}

Prior to a technical calculation
of the degree of $\phi$ that appears in the proposition \ref{Prop 3.1.1}, we need to do a bit of analysis about virtual fundamental class to figure out what kinds of tropical type $\pmb{\tau}$ will contribute to punctured invariants.

Note that there is an extreme situation in which the graph of $\pmb{\tau}$ has only one vertex carrying three legs $p, q, -r$. In this case, there is nothing we can glue and there is only one punctured log map to $W$ realizing it, which is the constant map to the point at which we impose point-constraint, in other words, the curve class $\beta$ of $\pmb{\tau}$ in this case is just $0$. So, this situation will not yield anything interesting by \cite{GS2}, Lemma 1.15. Henceforth, throughout this entire subsection, we are going to assume that the curve class $\beta$ is not $0$, which means we will no longer take this extreme situation into account.

The following lemma is one of the key observations to simplify the gluing problem.

\begin{lemma}\label{lem 3.2.1}
In the situation of lemma \ref{Lemma 3.1.1}, we further assume that both $x_1$ and $x_2$ lie in $\pmb{\tau}_2$, and let $\mathscr{M}_z(X,\pmb{\tau}_1):=\mathscr{M}(X,\pmb{\tau}_1)\times_{\underline{D}}z$ and $\mathfrak{M}_z^{\text{ev}}(\mathcal{X}, \pmb{\tau}_1):=\mathfrak{M}^{\text{ev}}(\mathcal{X}, \pmb{\tau}_1)\times_{\underline{D}}z$ be the moduli spaces appearing in the diagram of the proposition \ref{Prop 3.1.1}. Then 
$$\text{dim}(\mathfrak{M}_z^{\text{ev}}(\mathcal{X}, \pmb{\tau}_1))=-1.$$

Hence, $[\mathscr{M}_z(X,\pmb{\tau}_1)]^{\vir}=0$ and the moduli space $\mathscr{M}_z(X,\pmb{\tau_1})$ has no contribution at all to virtual class and curve counting, in other words, 
$$[\mathscr{M}_z(X,\pmb{\tau}_1)\times\mathscr{M}(X,\pmb{\tau}_2,z)]^{\vir}=
0.$$
\end{lemma}

\begin{remark}
Remember that after splitting $\pmb{\tau}$, $\pmb{\tau}_1$ acquires an extra punctured point $w$ and $\pmb{\tau}_2$ gets an extra punctured point $w'$ such that the contact orders at $w, w'$ are opposite to each other. 
\end{remark}
\begin{proof}
The key point is to compute the dimension of $\mathfrak{M}^{\text{ev}}(\mathcal{X},\pmb{\tau}_1)$. Firstly, since the only punctured point that belongs to $\pmb{\tau}_1$ is the one coming up after we split $\pmb{\tau}$ at the given nodal section, and by the stability condition (note that a point in $\mathfrak{M}^{\operatorname{ev}}(\mathcal{X},\pmb{\tau})$ corresponds to a punctured log map $f:C^\circ\rightarrow\mathcal{X}$ factors through a map $C^\circ\rightarrow X$, thus it makes sense to talk about stability condition), the contact order at that punctured point has to be positive. Recall that, by the definition of evaluation space, we have the following Cartesian diagram
\begin{center}
    \begin{tikzcd}
    \mathfrak{M}^{\text{ev}}(\mathcal{X},\pmb{\tau}_1)\arrow[r]\arrow[d] & X\arrow[d]\\
    \mathfrak{M}(\mathcal{X},\pmb{\tau}_1)\arrow[r] & \mathcal{X}
    \end{tikzcd}
\end{center}
By \cite{ACGS2},Prop.~3.28, we have
$\dim
(\mathfrak{M}(\mathcal{X},\pmb{\tau}_1))\le -2$. So, by the fiber diagram above, we have dim$(\mathfrak{M}^{\text{ev}}(\mathcal{X},\pmb{\tau}_1))\le \dim(X)-2$. Then by the definition of $\mathfrak{M}_z^{\text{ev}}(\mathcal{X}, \pmb{\tau}_1)$, we then get that dim$(\mathfrak{M}_z^{\text{ev}}(\mathcal{X}, \pmb{\tau}_1))=$ dim$(\mathfrak{M}^{\text{ev}}(\mathcal{X},\pmb{\tau}_1))-$ dim$(D)\le$ dim$(X)-2-(\text{dim}(X)-1)=-1$. 

On the other hand, obviously the map $\mathscr{M}_z(X,\pmb{\tau}_1)\rightarrow\mathfrak{M}_z^{\text{ev}}(\mathcal{X}, \pmb{\tau}_1)$ admits a perfect relative obstruction theory and $[\mathscr{M}_z(X,\pmb{\tau}_1)]^{\text{vir}}=\epsilon^{!}[\mathfrak{M}_z^{\text{ev}}(\mathcal{X}, \pmb{\tau}_1)]$ ($\epsilon$ is the map which appears in the proposition \ref{Prop 3.1.1}), then by Riemann-Roch, we know the relative virtual dimension is $\beta_1\cdot c_1(\Theta_{X/\Bbbk})=\beta_1\cdot(K_X+D)=\beta_1\cdot 0=0$, thus the virtual dimension of $\mathscr{M}_z(X,\pmb{\tau}_1)=-1$, so $[\mathscr{M}_z(X,\pmb{\tau}_1)]^{\text{vir}}=0$ since $\mathscr{M}_z(X,\pmb{\tau}_1)$ is a Deligne-Mumford stack. 
\end{proof}

\begin{corollary}\label{corollary 3.2.1}
In the situation of lemma \ref{lem 3.2.1}, we have 
$$[\mathscr{M}(X,\pmb{\tau},z)]^{\vir}=0.$$
\end{corollary}
\begin{proof}
According to \cite{GS2} and \cite{ACGS2}, we know $\mathscr{M}(X,\pmb{\tau},z), \mathscr{M}_z(X,\pmb{\tau}_1)$ and $\mathscr{M}(X,\pmb{\tau}_2,z)$ are all Deligne-Mumford stacks, also $\mathfrak{M}^{\text{ev}}(\mathcal{X},\pmb{\tau},z), \mathfrak{M}_z^{\text{ev}}(\mathcal{X}, \pmb{\tau}_1)$ and $\mathfrak{M}^{\text{ev}}(\mathcal{X},\pmb{\tau}_2,z)$ are all of pure dimensional and stratified by quotient stacks. Further, the lemma above tells us that $\text{dim}(\mathfrak{M}_z^{\text{ev}}(\mathcal{X}, \pmb{\tau}_1))=-1$ and  $\mathscr{M}_z(X,\pmb{\tau}_1)\rightarrow\text{dim}(\mathfrak{M}_z^{\text{ev}}(\mathcal{X}, \pmb{\tau}_1))$ has relative virtual dimension $\beta_1\cdot c_1(\Theta_{X/\Bbbk})=\beta_1\cdot(K_X+D)=0.$ Therefore, the vanishing of the virtual fundamental class in the corollary follows directly from Theorem A.16 of \cite{GS2}.
\end{proof}
\begin{remark}\label{rem 3.2.2}
By the lemma and the corollary we just saw, the virtual fundamental class of $\mathscr{M}(X,\pmb{\tau},z)$ will possibly not be 0 only if the $\pmb{\tau}_1$ contains at least one of $x_1$ and $x_2$. However, if $\pmb{\tau}_1$ contains both $x_1, x_2$ at the same time, in order to maintain the stability of our punctured logarithmic maps, the component that contains $x_{\text{out}}$ must contain an additional node, then we can further split $\pmb{\tau}_2$ at this additional node to get two new types of curves $\pmb{\tau}_2'$ and $\pmb{\tau}_3$ with $x_{\text{out}}\in\pmb{\tau}_2'$ and $\pmb{\tau}_3$ having not having any punctured points, which means we can split $\pmb{\tau}$ from the beginning so that one of the resulting class contains no punctured points, and then the virtual fundamental class corresponding to this type $\pmb{\tau}$ is actually $0$.
\end{remark}
\begin{corollary}\label{Cor 3.14}
The virtual fundamental class $[\mathscr{M}(X,\pmb{\tau},z)]^{\vir}$ non-zero
implies that $\pmb{\tau}$, as a tropical type of punctured log maps, is such that the vertex $V$ that contains the leg $x_{\text{out}}$ contains exactly one of legs $x_1$ and $x_2$ and any vertices other than $V$ get mapped to the origin by the corresponding tropical map.
\end{corollary}
\begin{proof}
First, if 
the graph $G$ of $\pmb{\tau}$ has at least two vertices mapping into the 
interior of the tropicalization $\Sigma(X)(=\mathbb{R}_{\geq 0})$ of $X$,
then the family of corresponding tropical maps is at least two-dimensional
as these vertices are free to move. This implies
that the virtual dimension of the moduli space $\mathscr{M}(X,\pmb{\tau},z)$ 
will be negative for such a tropical type $\pmb{\tau}$.

Since any irreducible component corresponding to a vertex mapping
into the interior of the tropicalization corresponds to a contracted
component, balancing holds at these vertices, see e.g., 
\cite{ACGS2}, Proposition 2.25.
From this, it follows that the component which contains the leg 
$x_{\text{out}}$ must also contain at least one of the legs $x_1, x_2$ 
Then the conclusion follows obviously from the corollary \ref{corollary 3.2.1} and \ref{rem 3.2.2}.
\end{proof}

We now note a virtual decomposition for the moduli spaces 
$\mathscr{M}(X,\pmb{\beta},z)$ giving the following formula, which is proved
in the same way as \cite{GS3}, \S6.

\begin{lemma} 
\label{lem:sum}
There is a decomposition
$$[\mathfrak{M}(\mathcal{X},\pmb{\beta},z)]=\sum_{\pmb{\tau}}\frac{m_{\pmb{\tau}}}{\operatorname{Aut}(\pmb{\tau})}[\mathfrak{M}(\mathcal{X},\pmb{\tau},z)],$$
where the sum is over all decorated types $\pmb{\tau}$ of punctured
tropical maps which are degenerations of $\pmb{\beta}$ and with one-dimensional
moduli space of tropical maps. Further, $m_{\pmb{\tau}}$ is the multiplicity
of the (union of) irreducible component of $\mathfrak{M}(\mathcal{X},\pmb{\beta},z)$
which is the image of $\mathfrak{M}(\mathcal{X},\pmb{\tau},z)$
in $\mathfrak{M}(\mathcal{X},\pmb{\beta},z)$. This
yields an equality on virtual fundamental classes via pull-back:
$$[\mathscr{M}(X,\pmb{\beta},z)]^\vir=\sum_{\pmb{\tau}}\frac{m_{\pmb{\tau}}}{\operatorname{Aut}(\pmb{\tau})}[\mathscr{M}(X,\pmb{\tau},z)]^\vir.$$
\end{lemma}

Next, we are going to deduce our main gluing theorem about 2-point Gromov-Witten invariants and punctured Gromov-Witten invariants. 
By Corollary \ref{Cor 3.14}, the only non-zero virtual cycle in the summation 
of Lemma \ref{lem:sum}
is given by the type satisfying the following assumption \ref{assumption 3.1}.

\begin{assumption}\label{assumption 3.1}
Type of curve class $\pmb{\tau}$ that we consider satisfies the basic setup described in the beginning of this section such that the vertex $V$ containing the puncturing leg $x_{\text{out}}$ contains exactly one of legs $x_1, x_2$. Moreover, we assume that any other vertices other than $V$ get mapped to the origin by the corresponding tropical map.
\end{assumption}

\begin{remark}
Once $\pmb{\tau}$ fulfills the assumption above, by the proof of lemma \ref{lem 3.2.1}, we can deduce that the moduli space $\mathscr{M}_z(X,\pmb{\tau}_1)$ has virtual dimension $0$, and if we assume further that the rational tail contains $x_2$, recall from the remark \ref{rem 3.2.1} that $\pmb{\tau}_1$ acquires an extra punctured point $w$ with contact order $q-r$ after the splitting, then we can define 2-point invariants with a point-constraint as follows
$$N_{p, q-r}:=\displaystyle\int_{[\mathscr{M}_z(X,\pmb{\tau}_1)]^{\text{vir}}}1.$$
When the rational tail contains $x_1$ instead of $x_2$, we can analogously define $N_{q, p-r}$. \\
Roughly speaking, $N_{p, q-r}$ virtually counts rational curves having curve class $\beta_1\in\text{H}_2(X,\mathbb{Z})$ and intersecting the divisor $D$ at two points $x,y$ with tangency order $p, q-r$ respectively in which we put a point constraint at $y$. More generally, we can define $N_{a,b}$ for any $a,b$ such that $a+b=\beta\cdot D$.
\end{remark}
\begin{proposition}\label{prop 3.3.1}
Let $\pmb{\tau}_2$ be the type after splitting $\pmb{\tau}$ which satisfies the assumption \ref{assumption 3.1} such that $x_2, x_{\text{out}}\in\pmb{\tau}_2$. Then the canonical projection map $\mathfrak{M}^{\text{ev}}(\mathcal{X},\pmb{\tau}_2,z)\rightarrow B\mathbb{G}_m^\dagger$ is logarithmically smooth.
\end{proposition}
\begin{proof}
Recall from the remark \ref{rem 2.5.2} that $\mathscr{P}(\mathcal{X},r)$ is an idealized log stack with the ideal sheaf $\overline{\mathcal{I}}\subset\overline{\mathcal{M}}_{\mathscr{P}(\mathcal{X},r)}$ given by $r^{-1}(\mathbb{Z}_{>0})$. In our case, since $\overline{\mathcal{M}}_Z$ is generated by $\mathbb{N}$, the ideal is generated by $\mathbb{Z}_{>0}$. For any parametrizing space $W\rightarrow\mathfrak{M}(\mathcal{X},\pmb{\tau}_2)$, note that the basic monoid on $\mathfrak{M}(\mathcal{X},\pmb{\tau}_2)$ is $\mathbb{N}$, so the pullback of the ideal to $W$ under the map $W\rightarrow\mathfrak{M}(\mathcal{X},\pmb{\tau}_2)\rightarrow\mathscr{P}(\mathcal{X},r)$ is also generated by $\mathbb{Z}_{>0}$.
Let $C^{\circ}\rightarrow \mathcal{X}$ the the map parameterized by $W$.

By the pre-stability condition,  $\overline{\mathcal{M}}_{C^\circ}
\subseteq \mathbb{N}\oplus \mathbb{Z}_{x_1}\oplus \mathbb{Z}_{x_2}\oplus 
\mathbb{Z}_{x_{\mathrm{out}}}$
is generated by $(1,0,0,0), (0,1,0,0)$,\\ $(0,0,1,0)$,
$(0,0,0,1)$, $(1,0,0,r-q)$ and $(1,0,-r,0)$. 
Note that the quantity $r-q$ is always negative by balancing condition of tropical maps (keep in mind that every vertex other than $V$ is mapped to the origin by the corresponding tropical map). Then, by the definition of puncturing log-ideal given in definition 2.30 of \cite{ACGS2}, once we project it into $\overline{\mathcal{M}}_W$, the puncturing log-ideal is generated by $\mathbb{Z}_{>0}$. Thus, the map $\text{ev}_{\mathcal{X}}: \mathfrak{M}(\mathcal{X},\pmb{\tau}_2)\rightarrow\mathscr{P}(\mathcal{X},r)$ is actually ideally strict (see section 1.3, chapter 3 of \cite{Og}).

By the theorem 3.15 of \cite{GS2}, we know that the map $\text{ev}_{\mathcal{X}}$ is idealized log smooth. Thus, it is in fact log smooth by \cite{Og}, IV invariant 3.1.22. Note that the projection $\mathfrak{M}^{\text{ev}}(\mathcal{X},\pmb{\tau}_2,z)\rightarrow B\mathbb{G}_m^\dagger$ is just a base change of $\text{ev}_{\mathcal{X}}$, hence it is log smooth. 
\end{proof}
\begin{corollary}\label{cor 3.3.3}
In the situation of the proposition \ref{prop 3.3.1}, the moduli space $\mathfrak{M}^{\text{ev}}(\mathcal{X},\pmb{\tau}_2,z)$ is a reduced algebraic stack. 
\end{corollary}
\begin{proof}
According to the tropical interpretation of this map (see lemma 3.22 of \cite{GS2}), smooth locally at a generic point $w$, we have a chart for this map 
\begin{center}
    \begin{tikzcd}
    \mathfrak{M}^{\text{ev}}(\mathcal{X},\pmb{\tau}_2,z)\arrow[r]\arrow[d] & {[\text{Spec}(\Bbbk[\mathbb{N}])/\mathbb{G}_m]}\arrow[d, "\eta"]\\
    B\mathbb{G}_m^{\dagger}\arrow[r] &{[\text{Spec}(\Bbbk[\mathbb{N}])/\mathbb{G}_m]}
    \end{tikzcd}
\end{center}
where the map between monoids corresponding to $\eta$ is sending $1$ to $\delta$ such that $\Sigma_1(v_{\text{out}})=\delta\cdot r$ (see remark \ref{rem 2.3.1} for the definition of $\Sigma_m$), here $v_{\text{out}}$ is the vertex corresponding to the component containing $x_{\text{out}}$ (don't forget that we only have one component in this case) and the tropical map 
\begin{center}
    \begin{tikzcd}
    \omega_{v_{\text{out}}}\arrow[d]\arrow[r, "\Sigma"] & \mathbb{R}_{\geq0}\\
    \sigma_y
    \end{tikzcd}
\end{center}
corresponds to the image $y$ of $w$ in $\mathfrak{M}(\mathcal{X},\pmb{\tau}_2)$. By the definition of basic monoid, we know that $\Sigma_m(v_{\text{out}})=m\cdot r$. Thus, $\delta=1$.\\
Note that by the proposition \ref{prop 3.3.1}, the map $\mathfrak{M}^{\text{ev}}(\mathcal{X},\pmb{\tau}_2,z)\rightarrow B\mathbb{G}_m^\dagger$ is log smooth, hence $\delta$ is literally just the multiplicity of the moduli space. So, $\mathfrak{M}^{\text{ev}}(\mathcal{X},\pmb{\tau}_2,z)$ is reduced and even smooth
over $B\mathbb{G}_m^\dagger$.
\end{proof}

\begin{lemma}\label{lem 3.2.2}
Let $\pmb{\tau}_2$ be the type after splitting $\pmb{\tau}$ which satisfies the assumption \ref{assumption 3.1} such that $x_{\text{out}}\in\pmb{\tau}_2$. Then the corresponding moduli space $\mathscr{M}(X,\pmb{\tau}_2,z)$ is a reduced point, i.e. $\mathscr{M}(X,\pmb{\tau}_2,z)=\operatorname{Spec}(\Bbbk)$.
\end{lemma}
\begin{proof}
At first, note that $\mathscr{M}(X,\pmb{\tau}_2,z)$ is a Deligne-Mumford stack according to \cite{GS2}. Any point $\text{Spec}(\Bbbk)\rightarrow\mathscr{M}(X,\pmb{\tau}_2,z)$ of $\mathscr{M}(X,\pmb{\tau}_2,z)$ corresponds to a (family of) punctured logarithmic map 
\begin{center}
    \begin{tikzcd}
    C^\circ\arrow[d]\arrow[r, "f"] & D\subset X\\
    \text{Spec}(\Bbbk)
    \end{tikzcd}
\end{center}
Note that $f$ is a constant map to $z\in D$ by the genericness of $z$, and $\underline{C}^\circ=\mathbb{P}^1$ has three punctured points. Thus the $\Bbbk$-points of $\mathscr{M}(X,\pmb{\tau}_2,z)$ is just the single point shown as above with trivial stabilizer, so $\mathscr{M}(X,\pmb{\tau}_2,z)$ is an algebraic space.

Furthermore, obviously there is a map $\mathscr{M}(X,\pmb{\tau}_2,z)\rightarrow\operatorname{Spec}(\Bbbk)$ which is of finite type, quasi-finite, and separated. So, by a theorem of Olsson in \cite{Ol2}, $\mathscr{M}(X,\pmb{\tau}_2,z)$ is a scheme. So, it is just a (possibly non-reduced) point.

However, the map $\mathscr{M}(X,\pmb{\tau}_2,z)\rightarrow\mathfrak{M}^{\text{ev}}(\mathcal{X},\pmb{\tau}_2,z)$ is smooth and $\mathfrak{M}^{\text{ev}}(\mathcal{X},\pmb{\tau}_2,z)$ is reduced by the corollary above. Hence, $\mathscr{M}(X,\pmb{\tau}_2,z)$ is just a reduced point. 
\end{proof}
\begin{theorem}\label{thm 3.2.1}
Let $(X,D)$ be a smooth log Calabi-Yau pair, i.e. $X$ is a smooth projective variety with a smooth divisor $D\subset X$. Let $\pmb{\tau}$ be the type of curve class fulfilling the assumption \ref{assumption 3.1}, and suppose that the vertex $V$ containing the leg $x_{\operatorname{out}}$ contains $x_2$ $(\text{resp.}\ x_1)$. Then the degree of the map $\phi$ which appears in the diagram of proposition \ref{Prop 3.1.1} has degree $q-r$ $(\text{resp.}\ p-r)$, in other words, the multiplicity $m_{\pmb{\tau}}=q-r$ $(resp.\ p-r)$.
\end{theorem}
\begin{proof}
In order to calculate the degree of $\phi$, we need to choose general 
points in $\mathfrak{M}_z^{\text{ev}}(\mathcal{X},\pmb{\tau}_1), \mathfrak{M}^{\text{ev}}(\mathcal{X},\pmb{\tau}_2,z)$ respectively and then compute the glued family of punctured logarithmic maps to see how many connected components will be produced after gluing.\\

\textbf{Step 1:} At first, let us choose a generic point $\underline{W_1}=\text{Spec}(\Bbbk)\rightarrow\mathfrak{M}_z^{\text{ev}}(\mathcal{X},\pmb{\tau}_1)$ of $\mathfrak{M}_z^{\text{ev}}(\mathcal{X},\pmb{\tau}_1)$ with the pullback logarithmic structure on $\underline{W_1}$ from the basic logarithmic structure on $\mathfrak{M}_z^{\text{ev}}(\mathcal{X},\pmb{\tau}_1)$. Without loss of generality, we can assume that the choice of point corresponds to a punctured log map $(C_1^\circ/W_1,\pmb{\tau}_1, f_1:C_1^\circ\rightarrow \mathcal{X})$ such that the source curve $C_1^\circ$ is smooth and isn't mapped into 
$\mathcal{D}$ by $f_1$. In this case, the basic log structure at this point is just trivial, i.e., $W_1=(\text{Spec}(\Bbbk),0)$.
(We can make such an assumption because such a kind of punctured log maps form a dense open subset of $\mathfrak{M}_z^{\text{ev}}(\mathcal{X},\pmb{\tau}_1)$ equipped with the basic log structure). \\

\textbf{Step 2:} Choose a generic point $\underline{W_2'}=\text{Spec}(\Bbbk)\rightarrow\mathfrak{M}^{\text{ev}}(\mathcal{X},\pmb{\tau}_2)$, so $\text{Spec}(\Bbbk)$ is equipped with the logarithmic structure pulled back from the basic logarithmic structure on $\mathfrak{M}^{\text{ev}}(\mathcal{X},\pmb{\tau}_2)$. Since $\pmb{\tau}_2$ contains $x_{\text{out}}$ having negative contact order, the whole component will be completely mapped into $D$, so the corresponding tropical map is parametrized by $\mathbb{R}_{\geq0}$. Hence, $W_2'=(\text{Spec}(\Bbbk),\mathbb{N})$.

Next, we need to compute $W_2:=W_2'\times_{\mathscr{P}(X,r)}B\mathbb{G}_m^{\dagger}$. Note that $W_2'$ parametrizes the punctured logarithmic map $(C^\circ/W_2',\pmb{\tau}_2, f':C^\circ\rightarrow \mathcal{X})$ where $C^\circ$ is just $\mathbb{P}^1$ and we chose $z\in D$ generically enough such that no rational curves pass through $z$ inside $D$, thus the $f'$ is the constant map to $z$. Thus, without loss of generality, we can assume that $$\widetilde{\mathscr{P}}(X,r)=(\text{Spec}(\Bbbk),\mathbb{N})\times B\mathbb{G}_m^{\dagger}=(B\mathbb{G}_m^{\dagger},\mathbb{N}\oplus\mathbb{N}).$$ 
Thus, by the definition of $\mathscr{P}(X,r)$, we get $\underline{\mathscr{P}}(X,r)=B\mathbb{G}_m$, and the ghost sheaf of the logarithmic structure is the sub-monoid $\{(m,rm)|m\in\mathbb{N}\}\subset\mathbb{N}\oplus\mathbb{N}$ which is isomorphic to $\mathbb{N}$. So, we have 
$$\mathscr{P}(X,r)=(B\mathbb{G}_m,\mathbb{N})$$
where the torsor associated to $1\in \mathbb{N}$ is the $r^{th}$ tensor
power of the universal torsor.
Then, if we impose the point-constraint, $W_2=W_2'\times_{B\mathbb{G}_m^{\dagger}}(B\mathbb{G}_m,\mathbb{N})$ with the corresponding maps between ghost sheaves 
\begin{center}
    \begin{tikzcd}
    \mathbb{N}\arrow[r, "\text{id}"]\arrow[d, "\cdot r"] & \mathbb{N}\\
    \mathbb{N}
    \end{tikzcd}
\end{center}
where the vertical multiplication-by-$r$ map is the map from the ghost sheaf of $\mathscr{P}(\mathcal{X},r)=(B\mathbb{G}_m,\mathbb{N})$ to the ghost sheaf of $B\mathbb{G}_m^{\dagger}$ and the horizontal map is given by the map from $W_2'$ to $B\mathbb{G}_m^{\dagger}$. Thus then push-out of the diagram above is $\mathbb{N}$. Thus, we have $W_2=(\text{Spec}(\Bbbk),\mathbb{N})$ with the corresponding map from the ghost sheaf of $\mathscr{P}(\mathcal{X},r)$ to the ghost sheaf of $W_2$ being the multiplication-by-$r$ map. Denote by $(C_2^\circ/W_2,\pmb{\tau}_2,f_2:C_2^\circ\rightarrow X)$ the family of punctured logarithmic maps parametrized by $W_2$.\\

\textbf{Step 3:} Let $w_1\in\pmb{\tau}_1$ and $w_2\in\pmb{\tau}_2$ be the two punctured points that we glue together to form the glued family of punctured logarithmic maps. Thus, $W_1^E=(\underline{W_1}, w_1^*\overline{\mathcal{M}}_{C_1^\circ, w_1})=(\text{Spec}(\Bbbk),\mathbb{N})$, and $W_2^E=(\underline{W_2}, w_1^*\overline{\mathcal{M}}_{C_2^\circ,w_2})$. So we need to figure out $\overline{\mathcal{M}}_{C_2^\circ,w_2}$ first. Note that we have a map $\overline{\mathcal{M}}_{X,z}=\mathbb{N}\rightarrow\overline{\mathcal{M}}_{C_2^\circ,w_2}\subset\mathbb{N}\oplus\mathbb{Z}$, and by the calculations we did in the step 2, we know this map is defined by sending the generator $1$ to $(r,r-q)$. By the pre-stability condition, $\overline{\mathcal{M}}_{C_2^\circ,w_2}$ is supposed to be generated by $\overline{\mathcal{M}}_{C_2}$ and the image of $\overline{\mathcal{M}}_{X,z}$. Hence, $\overline{\mathcal{M}}_{C_2^\circ,w_2}=\langle(1,0),(0,1),(r,r-q)\rangle=:R\subset\mathbb{N}\oplus\mathbb{Z}$. Therefore, $W_2^E=(\underline{W_2}, R)=(\text{Spec}(\Bbbk),R)$.\\

\textbf{Step 4:} This step is to saturate both $W_1^E$ and $W_2^E$. For $W_1^E$, there is nothing to saturate. So, $\widetilde{W}_1=W_1^E=(\text{Spec}(\Bbbk),\mathbb{N)}$. For $W_2^E$, the saturation $\widetilde{W}_2=\text{Spec}(\Bbbk)\times_{\text{Spec}(\Bbbk[R])}\text{Spec}(\Bbbk[R^{\text{sat.}}])$ where $R^{\text{sat.}}$ is the saturation monoid of $R$ inside $\mathbb{N}\oplus\mathbb{Z}$. Since $\Bbbk\otimes_{\Bbbk[R]}\Bbbk[R^\text{sat.}]\cong\Bbbk[R^{\text{sat.}}]/I$ where is the ideal generated by $(1,0), (0,1), (r,r-q)$ in $\Bbbk[R^{\text{sat.}}]$. Then, $\widetilde{W}_2=(\text{Spec}(\Bbbk[R^{\text{sat.}}]/I),R^{\text{sat.}})$.\\

\textbf{Step 5:} In this step, we compute the number of connected components of the glued family $\widetilde{W}_1\times_X^{\text{fs}}\widetilde{W}_2$ where the fiber product of $\widetilde{W}_1$ and $\widetilde{W}_2$ is taken in the category of fine and saturated logarithmic schemes. We know the corresponding maps between logarithmic structures:
\begin{center}
    \begin{tikzcd}
    \mathbb{N}\arrow[d]\arrow[r] & \mathbb{N}\\
    R^{\text{sat.}}
    \end{tikzcd}
\end{center}
where the vertical map sends $1$ to $(r,r-q)$ and the horizontal map sends $1$ to $q-r$.

We are hoping to apply Theorem 4.4 in \cite{Gro}, so we define a map as follows:
\begin{align*}
    \theta:=(\theta_1, -\theta_2):\mathbb{Z} & \longrightarrow\mathbb{Z}\oplus(R^{\text{sat.}})^{\text{gp}}\\
    1& \longmapsto (q-r,-r,q-r)
\end{align*}
where the superscript gp. means groupification. Then, by a direct computation, we assert that $\#(\text{coker}(\theta)_{\text{tor.}})=\text{g.c.d}(r, q-r)$ where $\#(\text{coker}(\theta)_{\text{tor.}})$ means the cardinality of the torsion elements of the cokernel of $\theta$ and $\text{g.c.d}(r, q-r)$ represents the greatest common divisor of $r$ and $q-r$.

So, by Theorem 4.4 in \cite{Gro}, since the fiber product is not empty, the connected component of $\widetilde{W}_1\times_X^{\text{fs}}\widetilde{W}_2$ is g.c.d($r,q-r$). This calculation is in fact giving us the degree of the map $\phi_{\text{red}}:\mathfrak{M}^{\text{ev}}(\mathcal{X},\pmb{\tau},z)_{\text{red}}\longrightarrow \mathfrak{M}_z^{\text{ev}}(\mathcal{X},\pmb{\tau}_1)\times \mathfrak{M}^{\text{ev}}(\mathcal{X},\pmb{\tau}_2,z)$ where the subscript red means the reduction of the original space.\\

\textbf{Step 6:} The last step is to figure out the non-reduced structure on $\mathfrak{M}^{\text{ev}}(\mathcal{X},\pmb{\tau},z)$. By Proposition B.2 in \cite{ACGS2}, the map $\mathfrak{M}^{\text{ev}}(\mathcal{X},\pmb{\tau},z)\rightarrow B\mathbb{G}_m^{\dagger}$ is logarithmically smooth. So, at a generic point, we can choose a chart for this map around the point shown as follows
\begin{center}
    \begin{tikzcd}
    \mathfrak{M}^{\text{ev}}(\mathcal{X},\pmb{\tau},z)\arrow[r]\arrow[d] & {[\mathbb{A}^1/\mathbb{G}_m]}\arrow[d, "\eta"]\\
    B\mathbb{G}_m^{\dagger}\arrow[r] & {[\mathbb{A}^1/\mathbb{G}_m]}
    \end{tikzcd}
\end{center}
and by the lemma 3.22 in \cite{GS2}, the corresponding map between monoids of $\eta$ is 
\begin{align*}
   \eta^*: \mathbb{N} & \longrightarrow \mathbb{N}\\
   1 & \longmapsto \delta
\end{align*}
here, $\delta$ is the smallest natural number such that $\Sigma_1(v_{\text{out}})=\delta\cdot r=l\cdot(q-r)$ for some integer $l$ in which $\Sigma_1$ is the corresponding tropical map and $v_{\text{out}}$ is the component containing $x_{\text{out}}$. Then, it is easy to see that $\delta=(q-r)/\text{g.c.d}(r,q-r)$. By the definition of logarithmic smoothness, $\delta$ is the multiplicity of $\mathfrak{M}^{\text{ev}}(\mathcal{X},\pmb{\tau},z)$.\\
Moreover, by the corollary \ref{cor 3.3.3}, $\mathfrak{M}^{\text{ev}}(\mathcal{X},\pmb{\tau}_2,z)$ is reduced, and $\mathfrak{M}_z^{\text{ev}}(\mathcal{X},\pmb{\tau}_1)$ is obviously reduced since it has no any non-trivial punctured points.\\
Hence, the degree of $\phi$ is $\text{g.c.d}(r,q-r)\cdot\delta=q-r$ as desired. 
\end{proof}

\begin{corollary}\label{corollary 3.2.3}
Let $X$ be a smooth log Calabi-Yau pair, then we have the following formula relating punctured Gromov-Witten invariants to 2-point Gromov-Witten invariants with point-constraint
$$N_{pqr}^{\pmb{\beta}}=(q-r)N_{p,q-r}+(p-r)N_{q,p-r}$$
where $r>0$.
\end{corollary}

\begin{proof}
Recall the part of the Cartesion diagram involved in $\phi$ and $\Tilde{\phi}$ from the proposition \ref{Prop 3.1.1}
\begin{center}
    \begin{tikzcd}
    \mathscr{M}(X,\pmb{\tau},z)\arrow[d]\arrow[r, "\Tilde{\phi}"] & \mathscr{M}_z(X,\pmb{\tau}_1)\times\mathscr{M}(X,\pmb{\tau}_2,z)\arrow[d, "\epsilon"]\\
    \mathfrak{M}^{\text{ev}}(\mathcal{X},\pmb{\tau},z)\arrow[r, "\phi"] & \mathfrak{M}_z^{\text{ev}}(\mathcal{X},\pmb{\tau}_1)\times \mathfrak{M}^{\text{ev}}(\mathcal{X},\pmb{\tau}_2,z)
    \end{tikzcd}
\end{center}
where $\phi$ is a finite surjective map. By the lemma \ref{lem 3.2.2}, $\mathscr{M}(X,\pmb{\tau}_2,z)=\text{Spec}(\Bbbk)$. There are two situations that could happen. 
\begin{itemize}
    \item If $\pmb{\tau}_2$ contains $x_1$, then by the theorem \ref{thm 3.2.1}, we know that the degree of $\Tilde{\phi}$ is $p-r$. Hence, we have $$\Tilde{\phi}_*[\mathscr{M}(X,\pmb{\tau},z)]^{\text{vir}}=(p-r)\cdot[\mathscr{M}_z(X,\pmb{\tau}_1)]^{\text{vir}}$$
    \item For the same reason, if $\pmb{\tau}_2$ contains $x_2$, we have 
    $$\Tilde{\phi}_*[\mathscr{M}(X,\pmb{\tau},z)]^{\text{vir}}=(q-r)\cdot[\mathscr{M}_z(X,\pmb{\tau}_1)]^{\text{vir}}$$
\end{itemize}
Then Lemma \ref{lem:sum} gives rise to the formula we want to prove.
\end{proof}

\begin{remark}
In \cite{FWY}, the authors define a new kind of relative Gromov-Witten invariants $\Tilde{N}_{pqr}^{\beta}$ using orbifold Gromov-Witten theory, in which the new theory also allows negative tangency order to occur. In \cite{You}, F. You found a same formula for their new relative invariants $\Tilde{N}_{pqr}^{\beta}$ and it is easy to argue that our invariants $N_{pqr}^{\pmb{\beta}}$ agree with theirs' $\Tilde{N}_{pqr}^{\beta}$. Moreover, he can proceed to calculate 2-pointed invariants using their invariants and $I$-function and $J$-function technique.
\end{remark}

\section{A general calculation for the invariants $N_{e-1,1}$}
\ \ \ Given a smooth log Calabi-Yau pair $(W,D)$ with $D$ nef, as the title of this section indicates, we are going to calculate the invariants $N_{e-1,1}$ by comparing it with closed Gromov-Witten invariants defined in \cite{LLW} and \cite{Cha}.
The methods we will apply here are the standard deformation to normal cone and degeneration formula. Therefore, let us quickly review those closed invariants and clarify the setup at first, and then get our hands on comparison of these invariants.

\subsection{Basic setup}
\ \ \ Denoting by $\omega_W$ the canonical bundle of $W$, we set $X=\mathbb{P}(\omega_W\oplus \mathcal{O})$, the projective completion of the canonical bundle with the structure map $p:X\rightarrow W$. Next, we are going to construct a degeneration of $X$. We let $\mathcal{X}=\text{Bl}_{D\times 0}(W\times \mathbb{A}^1)$ and let $\mathcal{D}$ be the strict transform of $D\times\mathbb{A}^1$ in $\mathcal{X}$. So, a general fiber of $\mathcal{X}\rightarrow \mathbb{A}^1$ is $W$ and the special fiber is $W\sqcup_D Y$, the union of $W$ and
$Y$. Here, $Y$ is the projective completion of the normal bundle to $D$ in $W$ over $D$, i.e., $Y=\mathbb{P}(N_{D/W}\oplus\mathcal{O})$.
We view $D\subset Y$ as the $0$-section $D_0$ of the normal bundle of $D$, 
and in particular, the normal 
bundle of $D$ in $Y$ is the dual of $N_{D/W}$. 

\medskip

Set $\mathcal{L}=\mathbb{P}(\mathcal{O}_{\mathcal{X}}(-\mathcal{D})\oplus \mathcal{O})$. It is easy to check for the degeneration $\pi:\mathcal{L}\rightarrow \mathbb{A}^1$ that a general fiber $\pi^{-1}(\{t\})$ is $X$ for all $t\neq0$ and the special fiber $\mathcal{L}_0$ is $W\times \mathbb{P}^1\sqcup_{D\times\mathbb{P}^1}\mathbb{P}(\mathcal{O}_Y(-D_{\infty})\oplus \mathcal{O})$. Moreover, there exists a natural projection map $q:\mathcal{L}/\mathbb{A}^1\rightarrow W\times\mathbb{A}^1$. For simplicity, we let $\mathcal{L}_X=W\times \mathbb{P}^1, \mathcal{L}_D=D\times \mathbb{P}^1$ and $\mathcal{L}_Y=\mathbb{P}(\mathcal{O}_Y(-D_{\infty})\oplus \mathcal{O})$. We endow $\mathcal{X}$ with the divisorial log structure given by the central fiber $\mathcal{X}_0$.

\begin{remark}
We abuse the notation $\mathcal{O}$ to mean structure sheaf or trivial line bundle on whatever space we have. So, the symbol $\mathcal{O}$ could mean different structure sheaves depending on context. 
\end{remark}

\begin{remark}
It is worthwhile mentioning that our idea in this subsection is analogous to that presented in \cite{vGR}. So, we will collect some results proven in that paper later on. 
\end{remark}

We consider the moduli space $\mathscr{M}_{0,1}(X, \beta+h)$ parametrizing genus 0 stable maps $f: C\rightarrow X$ with one marked point such that $f_{\ast}[C]=\beta+h$ where $\beta$ is a chosen effective curve class in $W$ (viewed as the 0-section sitting inside $X$) satisfying $c_1(X)\cdot \beta=0$ and $h$ is the fiber class $[\mathbb{P}^1]$. The invariants defined in \cite{Cha} are just 
$$\int_{[\mathscr{M}_{0,1}(X,\beta+h)]^{\text{vir}}}\mathrm{ev}^*[\text{pt}]$$
where $[\text{pt}]$ is a point class in $X$ of which the effect is to impose a point constraint on stable maps and $\mathrm{ev}$ is the evaluation map at
the marked point.
\begin{remark}
In fact, if we consider the corresponding moduli space with a point constraint $\mathscr{M}_{0,1}(X,\beta+h,\sigma):=\mathscr{M}_{0,1}(X,\beta+h)\times_X\{\text{pt}\}$, here $\sigma$ can be understood as a geometric point $\sigma:\{\text{pt}\}\rightarrow X$, then $\mathscr{M}_{0,1}(X,\beta+h,\sigma)$ possesses a virtual fundamental class and degree of the virtual class is precisely the same as the integration shown above. 
\end{remark}
\begin{lemma}\label{lemma 4.4}
Let $Q:\mathscr{M}(\mathcal{L}/\mathbb{A}^1,\beta+h)\rightarrow\mathscr{M}(W\times\mathbb{A}^1/\mathbb{A}^1,\beta)$ be the natural map from moduli space of stable log maps to the families $\mathcal{L}/\mathbb{A}^1$ to moduli space of log stable maps to $W\times\mathbb{A}^1/\mathbb{A}^1$, which is induced by $q:\mathcal{L}\rightarrow W\times\mathbb{A}^1$. Then, we have, for any $t\neq0\in\mathbb{A}^1$, the following equality 

$$(Q_t)_{\ast}[\mathscr{M}(X,\beta+h)]^{\text{vir}}=(Q_0)_{\ast}[\mathscr{M}(\mathcal{L}_0,\beta+h)]^{\text{vir}}$$
where $Q_0$ and $Q_t$ are the restrictions of the map $Q$ to the special fiber and any generic fiber respectively, and either side of the equality is a cycle class in $A_{\ast}(\mathscr{M}(W,\beta);\mathbb{Q})$.
\end{lemma}

 \begin{proof}
 Notice that we have, for any $t\neq0$, the following four Cartesian diagrams 
 \[\begin{tikzcd}
 \mathscr{M}(\mathcal{L}_0,\beta+h)\arrow[d, "Q_0"]\arrow[r] & \mathscr{M}(\mathcal{L}/\mathbb{A}^1, \beta+h)\arrow[d, "Q"]\arrow[leftarrow]{r} & \mathscr{M}(X,\beta+h)\arrow[d, "Q_t"]\\
 \mathscr{M}(W,\beta)\arrow[r]\arrow[d] & \mathscr{M}(W\times\mathbb{A}^1, \beta+h)\arrow[d,"p_2"]\arrow[leftarrow]{r} & \mathscr{M}(W,\beta)\arrow[d]\\
 \{0\}\arrow[hookrightarrow]{r}{i_0} & \mathbb{A}^1\arrow[hookleftarrow]{r}{i_t}  & \{t\}
 \end{tikzcd}
 \]
 So, by commutativity of Gysin pullback with proper pushforward, we have 
 $$(Q_0)_{\ast}[\mathscr{M}(\mathcal{L}_0,\beta+h)]^{\text{vir}}=i_0^{!}Q_{\ast}[\mathscr{M}(\mathcal{L}/\mathbb{A}^1, \beta+h)]^{\text{vir}}=i_t^{!}Q_{\ast}[\mathscr{M}(\mathcal{L}/\mathbb{A}^1, \beta+h)]^{\text{vir}}=(Q_t)_{\ast}[\mathscr{M}(X,\beta+h)]^{\text{vir}}$$
 in which the middle equality holds because the family $p_2$ is a trivial family. 
 \end{proof}

\subsection{Degeneration formula}
\ \ \ Since our method involves calculating invariants of the special fiber, which is a normal crossing union of two spaces, the appearance of the degeneration formula becomes inevitable. In this subsection, we forget our setup described above for the time being and give a more general description of the formula dealing with a normal crossing union of two spaces. Let $X$ be a logarithmically smooth variety over a standard log point with irreducible components $X_0$ and $Y_0$. The central fiber of the degeneration described in the previous subsection is an example. A degeneration formula for stable log maps was first given by B. Kim, H. Lho and H. Ruddat in \cite{KLR}. In principle, their degeneration formula is enough for the sake of our calculations. However, as we are going to spend more time later dealing with moduli spaces with a point constraint, a decomposition formula for moduli spaces with
point constraints given in \cite{ACGS1} can be applied first in order to simplify the terms in Kim et al's degeneration formula. We refer readers to \cite{KLR} and \cite{ACGS1} for further details about the degeneration formula and decomposition formula respectively. 

We recall that given a curve class $\beta$ in $X$, we can consider moduli spaces $\mathscr{M}(X,\pmb{\tau})$ of stable log maps marked by $\pmb{\tau}=(\tau,\beta)$ where $\tau=(G,\pmb{g},\sigma, \pmb{u})$, a decorated type of a rigid tropical map. See \cite{ACGS1} for the definition of rigid tropical map. However, rigid tropical maps in our setting becomes much easier to describe. The proposition 5.1 in \cite{ACGS1} actually tells us that a decorated type $\pmb{\tau}$ of tropical maps is rigid in our setting if and only if every vertex will represent an irreducible component which gets mapped to either $X_0$ or $Y_0$, and we call a vertex $V$ an $X_0$-vertex if the corresponding component gets mapped into $X_0$ and likewise for $Y_0$-vertex. Furthermore, no edge connects two vertices which are both $X_0$-vertices or $Y_0$ vertices,
i.e., the endpoints of any edge will be of different type. On the other hand, a choice of decorated rigid tropical map in our situation is exactly what Jun Li terms an \textit{admissible triple} in \cite{Li1} and \cite{Li2}. More precisely, giving a rigid tropical map is equivalent to giving a \textit{bipartite} graph $\Gamma$ in which the edges are enumerated $e_1, e_2.\dotso, e_r$ with each edge decorated with a positive integer $w_e$ and each vertex $V$ is decorated with  a set $n_V$ thought of as the set of markings and a class $\beta_V$ that is an effective curve class in either $X_0$ or $Y_0$. For more details, readers can refer to \cite{Li1}, \cite{Li2} and \cite{vGR}.

So, we are actually free to go back and forth between our tropical language and Jun Li's degeneration formula. 


We incorporate a point constraint by defining 
$$\mathscr{M}(X,\pmb{\tau},s):=\mathscr{M}(X,\pmb{\tau})\times_X\{\text{pt\}}, \mathscr{M}(X,\beta,s):=\mathscr{M}(X,\beta)\times_X\{\text{pt\}}$$
where $s:\{\text{pt}\}\rightarrow X$ is a geometric point. Note that there is a finite map $j_{\pmb{\tau}\ast}:\mathscr{M}(X,\pmb{\tau},s)\rightarrow\mathscr{M}(X,\beta,s)$. Now we have the following decomposition for a point condition. 

\begin{theorem}
Suppose $X$ is logarithmically smooth. Then \[
[\mathscr{M}(X,\beta,s)]^\vir=\sum_{\pmb{\tau}=(\tau,\beta)}\frac{m_{\tau}}{|\text{Aut}(\tau)|}j_{\pmb{\tau}\ast}[\mathscr{M}(X,\pmb{\tau},s)]^{\vir}
\]
\end{theorem}

Furthermore, we will apply Kim, Lho and Ruddat's degeneration formula for the spaces $\mathscr{M}(X,\pmb{\tau}, s)$. Note that there is a diagram 
\begin{center}

\begin{tikzcd}
\mathscr{M}(X,\pmb{\tau})\arrow[rd, "\phi"]\\
& \bigodot_V\mathscr{M}_V\arrow[r]\arrow[d] & \prod_V\mathscr{M}_V\arrow[d]\\
& \prod_eD\arrow[r,"\Delta"] & \prod_V\prod_{V\in e}D\times D
\end{tikzcd}
\end{center}
where $V$'s and $e$'s represent the vertices and the edges in $\pmb{\tau}$ respectively and the space $\bigodot_V\mathscr{M}_V$ is defined by the Cartesian square diagram.
By \cite{KLR}, Equation (1.4), we know that $\phi$ is a finite etale map and the degree of $\phi$ is $\displaystyle\frac{\prod_ew_e}{\operatorname{l.c.m}(w_e)}$.
So we have the following equality 

$$[\mathscr{M}(X,\pmb{\tau})]^{\text{vir}}=\displaystyle\frac{\prod_ew_e}{\operatorname{l.c.m}(w_e)}\phi^\ast\Delta^!\prod_V[\mathscr{M}_V]^{\text{vir}},$$
where $w_e$ is the contact order to the divisor $\underline{D}$ at the relative marking corresponding to $e$. 

Moreover, we can impose point constraints at vertices of $\pmb{\tau}$ and an analogous decomposition will hold as well. 

\subsection{The main comparison}
\label{subsec: comparison}
\ \ \ Now, we are in position to deduce our comparison result by applying the degeneration formula to $\mathcal{L}_0$ which is the union of $\mathcal{L}_X=W\times\mathbb{P}^1$ and $\mathcal{L}_Y=\mathbb{P}(\mathcal{O}_Y(-D_{\infty})\oplus\mathcal{O})$ along the shared divisor $\mathcal{L}_D=D\times\mathbb{P}^1$. Then we need to impose a point condition for the family. It is easy to choose a section $s:\mathbb{A}^1\rightarrow\mathcal{L}$ of $\pi$ such that $s(t)$ lies on the fiber of $X$ for $t\neq 0$ and $s(0)\in \mathcal{L}_Y$ with a fiber $\mathbb{P}^1$ over $Y$ passing though $s(0)$. For instance, let us select any point $z\in D$, according to the construction of the space $\mathcal{L}$, we can choose the section $s':\mathbb{A}^1\rightarrow W\times\mathbb{A}^1$ given by $t\mapsto (z,t)$. We then take the strict transform of
this section to the blow-up $\mathcal{X}$, and then 
choose a general section of the $\mathbb{P}^1$-bundle $\mathcal{L}\rightarrow
\mathcal{X}$ over this strict transform.

\medskip

From now on, we can always assume that $s(0)\in\mathcal{L}_Y$, and then there exists a fiber of $\mathcal{L}_Y\rightarrow Y$ passing through $s(0)$. The class of the fiber is denoted by $h$.
Put in another way, when we have a decorated tropical type of rigid tropical maps $\pmb{\tau}$ with the graph $G$, we can always assume that there is only some $\mathcal{L}_Y$-vertex $V$ such that it carries the point constraint and the associated curve class $\beta_V$ contains $h$. Then we can consider the related moduli spaces with this point condition. The key step is to work out what kind of graphs will non-trivially contribute to virtual fundamental class in the degeneration formula. The following theorem basically deals with this problem. 

\medskip

\begin{theorem}\label{theorem 4.6}
For a decorated type of rigid tropical maps $\pmb{\tau}$, let $P:=Q_0$ from Lemma \ref{lemma 4.4}. Then we have $P_{\ast}j_{\pmb{\tau}\ast}[\mathscr{M}(\mathcal{L}_0,\pmb{\tau},s)]^{\vir}=0$ unless the graph $G$ of $\pmb{\tau}$ is the following
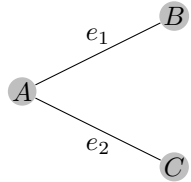
\[\begin{tikzpicture}[shorten >=1pt,->]
  \tikzstyle{vertex}=[circle,fill=black!25,minimum size=0.2pt,inner sep=0.2pt]
  \node[vertex] (G_1) at (-2,0) {$A$};
  \node[vertex] (G_2) at (0,1)   {$B$};
  \node[vertex] (G_3) at (0,-1) {$C$};
  
  \draw (G_1)--(G_2)node[midway, above]{$e_1$} (G_1)--(G_3)node[midway, below]{$e_2$}--cycle;
  
  \end{tikzpicture}\]
  with A an $\mathcal{L}_X$-vertex, $B$ and $C$ both $\mathcal{L}_Y$-vertices and furthermore,
$\omega_{e_1}+\omega_{e_2}=\beta\cdot\mathcal{L}_D$, $\beta_A=\beta$, $\beta_B$ is the $\omega_{e_1}$ times the class of a fiber of $p:Y\rightarrow D$ plus $h$ (thus carrying the point constraint) and $\beta_C$ is $\omega_{e_2}$ times the class of a fiber of $p:Y\rightarrow D$.
\end{theorem}
 Before jumping to the proof directly, we need a lemma as preparation. 
 
 \begin{lemma}\label{lemma 4.10}
 For a decorated type of rigid tropical maps $\pmb{\tau}$ with the graph $G$ of
$\pmb{\tau}$ having an $\mathcal{L}_X$-vertex $V$, let $r$ be the number of 
edges adjacent to $V$ connecting to $\mathcal{L}_Y$-vertices. Then $P_{\ast}j_{\pmb{\tau}\ast}[\mathscr{M}(\mathcal{L}_0,\pmb{\tau},s)]^{\vir}=0$ if $r>2$.
 \end{lemma}
 \begin{proof}
By the choice of the section $s$, we know that $\beta_V=\beta$ (no fiber class attached to $V$). Since a map from a proper curve to $\mathbb{P}^1$ which is not surjective is just a constant map, by separating out the factors for $V$, the gluing diagram factors as follows

\[\begin{tikzcd}
\mathscr{M}_V\times_{D_0}\bigodot_{V'\neq V}\mathscr{M}_{V'}\arrow[d, "ev"]\arrow[r] & \mathscr{M}_V\times_{D_0}\prod_{V'\neq V}\mathscr{M}_{V'}\arrow[d,"ev"]\\
D^r\times\mathbb{P}^1\arrow[r, "\Delta'"]\arrow[d, "id\times diag=:\delta"] & (D^r\times \mathbb{P}^1)^2\arrow[d,"id\times diag"]\\
(D\times \mathbb{P}^1)^r\arrow[r,"\Delta"] & (D\times\mathbb{P}^1)^{2r}
\end{tikzcd}
\]
where the two squares are both Cartesian.

Let $N$ denote the normal bundle of the embedding $\Delta$ which has rank $r(\text{dim}D+1)$ and $N'$ denote that of $\Delta'$ which has the rank $r\text{dim}D+1$. Set $E:=(\delta^{\ast}N)/N'$ which is of rank $r-1$. Let $c_{r-1}(E)$ be its top Chern class. For any $k$ and $\alpha\in A_k(\mathscr{M}_V\times_{D_0}\prod_{V'\neq V}\mathscr{M}_{V'})$, the excess intersection formula says that 

$$\Delta^!\alpha=c_{r-1}(E)\cap(\Delta')^!\alpha.$$
Note that $c_{r-1}(E)=0$ when $r\geq 3$ since the bundle $E$ is in fact the pullback of the corresponding bundle from the diagram 
\[\begin{tikzcd}
\mathbb{P}^1\arrow[r,"\Delta'"]\arrow[d,"diag"] & \mathbb{P}^1\times \mathbb{P}^1\arrow[d,"diag"]\\
(\mathbb{P}^1)^r\arrow[r,"\Delta"] & (\mathbb{P}^1\times\mathbb{P}^1)^{r}
\end{tikzcd}
\]
and taking Chern classes commutes with the pullback operation. Then applying this to the virtual cycle $\alpha=[\mathscr{M}_V]^\vir\times_{D_0}\prod_{V'\neq V}[\mathscr{M}_{V'}]^\vir$ gives the conclusion of this lemma. 
 \end{proof}


 \begin{proof}[\textit{Proof of the theorem}]
 For a decorated type of rigid tropical maps $\pmb{\tau}$ with the graph $G$, let us collect what is implied for $G$ if the pushforward to $\mathscr{M}(W,\beta)$ of the corresponding virtual cycles is non-trivial. Firstly, by the lemma \ref{lemma 4.10}, each of the $\mathcal{L}_X$-vertices of the graph $G$ has no more than 2 adjacent edges.

 Let $V$ be any $\mathcal{L}_Y$-vertex in $G$. Firstly, by 
\cite[Proposition 5.3]{vGR},
we know that $\pi_{Y\ast}\beta_V$ must be a multiple of the fiber class of the projective bundle $Y\rightarrow D$ where $\pi_Y$ is just the natural projection $\mathcal{L}_Y\rightarrow Y$, otherwise $P_\ast j_{\pmb{\tau}\ast}[\mathscr{M}(\mathcal{L}_0,\pmb{\tau},s)]^{\text{vir}}=0$. Hence, the curve class $\beta_V$ associated to the vertex $V$ is either $\beta+h$ or $\beta$ in which $\beta$ is a multiple of the fiber class of the projective bundle $Y\rightarrow E$ and $h$ is the fiber class of the projective bundle $\mathcal{L}_Y\rightarrow Y$. 
 
 \medskip
 
 Furthermore, we need to argue that $V$ has only a single adjacent edge to it. This is easily verified using \cite[Lemma 5.4]{vGR}. Indeed, let $\mathscr{M}^\circ:=\prod_{V'\neq V}\mathscr{M}_{V'}$, then the evaluation map from $\mathscr{M}_{G_V}(\mathcal{L}_Y(\text{log}\mathcal{L}_D),\beta_V)\times_{(D\times \mathbb{P}^1)^{r_V}}\mathscr{M}^\circ$ to $\mathscr{M}(W,\beta)$ factors through $D\times_{(D\times \mathbb{P}^1)^{r_V}}\mathscr{M}^\circ$ in either case of $\beta_V=\beta$ or $\beta_V=\beta+h$. Therefore, by \cite[Lemma 5.4]{vGR}, we can immediately conclude that there are at most 2 $Y_0$-vertices existing in $\pmb{\tau}$ or otherwise $P_\ast j_{\pmb{\tau}\ast}[\mathscr{M}(X,\pmb{\tau},s)]^{\text{vir}}=0$.
 \end{proof}
 
 \begin{proposition}
 For a decorated type of rigid tropical maps $\pmb{\tau}$ with the graph $G$ as in the theorem, we have $P_{\ast}j_{\pmb{\tau}\ast}[\mathscr{M}(\mathcal{L}_0),\pmb{\tau},s)]^{\text{vir}}=0$ unless $w_{e_1}=1$ and $w_{e_2}=e-1$.
 \end{proposition}
 
 \begin{proof}
 Let $V$ denote the vertex $B$ which carries the point constraint, and the curve class $\beta_V=\beta+h$ where $\beta$ is the fiber class of $Y\rightarrow D$ and $h$ is the fiber class of $\mathcal{L}_Y\rightarrow Y$. Consider the surface $S$ which is obtained by restricting $\mathcal{L}_Y$ onto a fiber of $Y\rightarrow D$ such that it contains the point constraint $s(0)$.
In fact, it is not so hard to see that $S$ is just the Hirzebruch surface $F_1$. Indeed, recall that $\mathcal{L}_Y=\mathbb{P}(\mathcal{O}_Y(-D_{\infty})\oplus\mathcal{O})$, and note that the restriction of $\mathcal{O}_Y(-D_{\infty})\oplus\mathcal{O}$ to any fiber of $Y\rightarrow D$ is actually $\mathcal{O}_{\mathbb{P}^1}(-1)\oplus\mathcal{O}_{\mathbb{P}^1}$ since the intersection of $-D_\infty$ and the fiber is $-1$.

Note that any curve realizing the curve class $\beta+h$ and containing the point $s(0)$ will lie in this surface $S$. 
 
 Assume that $w_{e_1}\geq 2$. Let $C$ be any curve which lies in $S$ and realizes the class $\beta$, that is, $C$ is a fiber of $Y\rightarrow D$. Then for any curve in $C'\subset S$ such that $[C']=w_{e_1}\beta+h$, we find that $C\cdot C'=\beta\cdot(w_{e_1}\beta+h)=-w_{e_1}+1<0$. Therefore, $C'$ has to contain a fiber of $Y\rightarrow D$ which is contained in $S$ as its one irreducible component. The evaluation map from $\mathscr{M}_{V}$ to $\mathcal{L}_D$ is a constant map. Hence, the vanishing result follows immediately by applying the degeneration formula.
 \end{proof}

 \begin{theorem}\label{theorem 4.9}
 $(-1)^{e-1}\cdot(e-1)\cdot p_\ast[\mathscr{M}(X,\beta+h, s)]^{\text{vir}}=[\mathscr{M}(W(\log D),\beta,s)]^{\text{vir}}$ where $\beta$ is an effective curve class in $W$ and $h$ is the fiber class of $p: X\rightarrow W$.
 \end{theorem}
 
 \begin{proof}
 The proof is now just an application of degeneration formula. By the above theorem and the degeneration formula, we can conclude that 

$$[\mathscr{M}(\mathcal{L}_0,\beta+h, s)]^{\text{vir}}=\frac{m_{\pmb{\tau}}}{|\text{Aut}(\pmb{\tau})|}j_{\pmb{\tau}\ast}\phi^{\ast}\Delta^!\prod_{V\in\pmb{\tau}}[\mathscr{M}_V]^{\text{vir}}$$
where the decorated tropical type of rigid tropical maps $\pmb{\tau}$ is exactly the one depicted in the theorem \ref{theorem 4.6}. So, in our situation, $|\text{Aut}(\pmb{\tau})|=1$ and $m_{\pmb{\tau}}=e-1$.

\medskip

Let $P:\mathcal{L}_0\rightarrow W$ be the natural projection. Then, by the lemma \ref{lemma 4.4}, we have 
\[
p_\ast[\mathscr{M}(X,\beta+h,s)]^{\text{vir}}=(e-1)\cdot P_{\ast}j_{\pmb{\tau}\ast}\phi^{\ast}\Delta^!\prod_{V\in\pmb{\tau}}[\mathscr{M}_V]^{\text{vir}}
\]
By \cite[Proposition 2.4]{vGR}, we know that the right hand side becomes $(e-1)\cdot\displaystyle\frac{(-1)^{e}}{(e-1)^2}[\mathscr{M}(W(\text{log}D),\beta,s)]^{\text{vir}}$ where the space $\mathscr{M}(W(\text{log}D),\beta,s)$ is exctly the moduli space of genus 0 stable logarithmic maps to $W$ with the curve class $\beta$ such that image of map intersects $D$ at one specified point to tangent order 1 and at one unspecified point to tangent order $e-1$, that is, we have an equality of cycles
$$(-1)^{e-1}\cdot(e-1)\cdot p_\ast[\mathscr{M}(X,\beta+h, s)]^{\text{vir}}=[\mathscr{M}(W(\text{log}D),\beta,s)]^{\text{vir}}.$$
\end{proof}
As a direct consequence of theorem \ref{theorem 4.9}, we find:

 \begin{corollary}\label{corollary 4.2.1}
 $(-1)^{e}\cdot(e-1)\cdot n_{\beta+h}=N_{e-1,1}$ where $\beta$ is an effective curve class in $W$ and $h$ is the fiber class of $X\rightarrow W$.
 \end{corollary}
 
\begin{remark}
All those invariants $n_{\beta+h}$ have been calculated by Siu-Cheong Lau in \cite{Lau}. Therefore, after identifying $N_{3d-1,1}$ with $n_{\beta+h}$ using the equality in Corollary \ref{corollary 4.2.1}, we can calculate all 2-pointed relative GW invariants for any degree $d$. 
\end{remark}

\section{Calculations for $(\mathbb{P}^2, D)$}
Throughout this section, our smooth log Calabi-Yau pair is always $(\mathbb{P}^2, D)$ where $D$ is a smooth cubic curve.

\subsection{A relation between invariants $N_{a,b}$}
\ \ \ In this subsection, we will investigate how those punctured invariants and 2-point relative invariants with a point-constraint abstractly determine each other. Besides the formula
we proved in Corollary \ref{corollary 3.2.3}, the other key intermediate tool we are going to use is the  ``degree 0 part of relative quantum cohomology ring" defined by Gross and Siebert in \cite{GS2}.\\

Roughly speaking, our degree 0 part of relative quantum cohomology is a $\mathbb{Q}[t]-$algebra structure on 
$$R:=\bigoplus_{p=0}^{\infty}\theta_p\mathbb{Q}[t],$$
where $\theta_p$ can be viewed as a symbol, and the multiplication law is 
$$\theta_p\cdot\theta_q:=\sum_{r=0}^{\infty}\sum_{d=0}^{\infty}N_{pqr}^dt^d\theta_r.$$
Here, for the precise definition of $N_{pqr}^d$, see definition \ref{def 3.1.1}. However, roughly speaking, it is the virtual count of degree $d$ genus zero stable maps with three punctured points $x_1, x_2, x_{\text{out}}$, with $x_1$ having contact order $p$ with $D$, $x_2$ having contact order $q$, and $x_{\text{out}}$ having contact order $-r$ with $D$ such that $x_{\text{out}}$ gets mapped to a generically fixed point $z\in D$.\\

Then, we can substitute those 2-point invariants for $N_{pqr}^d$ by means of the formula in corollary \ref{corollary 3.2.3}, and we have 
$$\theta_p\cdot\theta_q=\sum_{d=0}^{\infty}N_{pq0}^dt^d\theta_0+\sum_{r=1}^{\infty}\sum_{d=0}^{\infty}((q-r)N_{p,q-r}+(p-r)N_{q,p-r})t^d\theta_r.$$

Notice that firstly, $N_{p,i}$ are the number of rational curves with two marked points, one tangent to order $p$ at an unspecified point of $D$ and one tangent to order $i$ at a specified point of $D$. So, if $p\leq0$ or $i\leq0$, we have $N_{p,i}=0$. Also, note that these rational curves must be of degree $(p+i)/3$.\\

Moreover, by the lemma 1.15 in \cite{GS2}, we have 
\[N_{pqr}^0=\begin{cases}
1 & r=p+q\\
0 & \text{otherwise}.
\end{cases}
\]
So, combining all aforementioned results and separating the degree 0 term, we can write the multiplication law as follows
$$\theta_p\cdot\theta_q=\theta_{p+q}t^0+N_{pq0}^{(p+q)/3}t^{(p+q)/3}\theta_0+\sum_{r=1}^{\text{max}(p,q)}[(q-r)N_{p,q-r}+(p-r)N_{q,p-r}]t^{(p+q-r)/3}\theta_r.$$
Note that the right-hand side is actually a finite sum.\\

Let's see what relations we can get from associativity. We can calculate the product $\theta_p\cdot\theta_q\cdot\theta_r$ in two ways. Recall that $\theta_0$ is the identity in the ring, so we will not gain anything if one of $p, q, r$ are 0. Thus we assume $p, q, r>0$.\\
We have 
\begin{align*}
   (\theta_p\cdot\theta_q)\cdot\theta_r &=\left(\theta_{p+q}+N_{pq0}^{(p+q)/3}t^{(p+q)/3}\theta_0+\sum_{s=1}^{\text{max}(p,q)}[(q-s)N_{p,q-s}+(p-s)N_{q,p-s}]t^{(p+q-s)/3}\theta_s\right)\cdot\theta_r\\
    &=\theta_{p+q+r}+N_{p+q,r,0}^{(p+q+r)/3}t^{(p+q+r)/3}\theta_0+\sum_{s'=1}^{\text{max}(p+q,r)}[(r-s')N_{p+q,r-s'}+(p+q-s')N_{r,p+q-s'}]t^{(p+q+r-s')/3}\theta_{s'}\\
    &+N_{pq0}^{(p+q)/3}t^{(p+q)/3}\theta_r+\sum_{s=1}^{\text{max}(p,q)}[(q-s)N_{p,q-s}+(p-s)N_{q,p-s}]t^{(p+q-s)/3}\cdot\\
    &\cdot\left(\theta_{s+r}+N_{sr0}^{(s+r)/3}t^{(s+r)/3}\theta_0+\sum_{w=1}^{\text{max}(s,r)}[(r-w)N_{s,r-w}+(s-w)N_{r,s-w}]t^{(s+r-w)/3}\theta_{w}\right)
\end{align*}
and associating the other way gives:
\begin{align*}
    \theta_p\cdot(\theta_q\cdot\theta_r)&=\theta_p\cdot\Bigg(\theta_{q+r}+N_{qr0}^{(q+r)/3}t^{(q+r)/3}\theta_0+\sum_{s=1}^{\text{max}(q,r)}[(r-s)N_{q,r-s}+(q-s)N_{r,q-s}]t^{(q+r-s)/3}\theta_s\Bigg)\\
    &=\theta_{p+q+r}+N_{p,q+r,0}^{(p+q+r)/3}t^{(p+q+r)/3}\theta_0+\sum_{s'=1}^{\text{max}(p,q+r)}[(q+r-s')N_{p,q+r-s'}+(p-s')N_{q+r,p-s'}]t^{(p+q+r-s')/3}\theta_{s'}\\
    &+N_{qr0}^{(q+r)/3}t^{(q+r)/3}\theta_p+\sum_{s=1}^{\text{max}(q,r)}[(r-s)N_{q,r-s}+(q-s)N_{r,q-s}]t^{(q+r-s)/3}\cdot\\
    &\cdot\Bigg(\theta_{p+s}+N_{ps0}^{(p+s)/3}t^{(p+s)/3}\theta_0+\sum_{w=1}^{\text{max}(p,s)}[(s-w)N_{p,s-w}+(p-w)N_{s,p-w}]t^{(p+s-w)/3}\theta_w\Bigg)
\end{align*}
   Let us extract the highest degree contribution to the products above. Suppose that $p+q+r\equiv i$ mod $3$, with $i\in\{0,1,2\}$. We have the following possibilities:
   \begin{itemize}
       \item If $i=0$, then we get a coefficient of $t^{(p+q+r)/3}$ of the two products being 
       \begin{align}
        &\ \ \ N_{p+q,r,0}^{(p+q+r)/3}+\sum_{s=1}^{\text{max}(p,q)}[(q-s)N_{p,q-s}+(p-s)N_{q,p-s}]N_{sr0}^{(s+r)/3}\nonumber\\
        &=N_{p,q+r,0}^{(p+q+r)/3}+\sum_{s=1}^{\text{max}(q,r)}[(r-s)N_{q,r-s}+(q-s)N_{r,q-s}]N_{ps0}^{(s+p)/3}
    \end{align}
    
    \item If $i=1$ or $2$, then the highest degree power of $t$ is $t^{(p+q+r-i)/3}$, and the coefficient of this is $\theta_i$ times the following:
    \begin{align}
        &[(r-i)N_{p+q,r-i}+(p+q-i)N_{r,p+q-i}]+N_{pq0}^{(p+q)/3}\delta_{ri}\nonumber\\
        &+\sum_{s=1}^{\text{max}(p,q)}[(q-s)N_{p,q-s}+(p-s)N_{q,p-s}][\delta_{s,i-r}+(r-i)N_{s,r-i}+(s-i)N_{r,s-i}]\nonumber\\
        &=[(q+r-i)N_{p,q+r-i}+(p-i)N_{q+r,p-i}]+N_{qr0}^{(q+r)/3}\delta_{pi}\nonumber\\
        &+\sum_{s=1}^{\text{max}(q,r)}[(r-s)N_{q,r-s}+(q-s)N_{r,q-s}][\delta_{s,i-p}+(s-i)N_{p,s-i}+(p-i)N_{s,p-i}].
    \end{align}
   \end{itemize}
   Note that we take a convention that
$N_{pq0}^{(p+q)/3}$ and $N_{p,q}$ are both zero if $3$ doesn't divide
$p+q$.

   Before showing any relations between these invariants, there is an easy but interesting observation shown in the following lemma. 
   \begin{lemma}\label{lem 4.1.1}
For any positive integer $d$, we have $N_{3d-1,1}=(3d-1)^2N_{1,3d-1}$.
\end{lemma}
\begin{proof}
The equality holds for the corresponding orbifold Gromov-Witten invariants as a direct corollary of the formula given by Cadman and Chen in \cite{CC}, and according to the comparison result proven by Abramovich-Cadman-Wise in \cite{ACW}, the equality holds in our setting. 
\end{proof}
By the equations (4.1) and (4.2) we derived above, we can prove the following proposition.
\begin{proposition}
For $a+b=3d$, the invariants $N_{ab0}^d$ and $N_{a,b}$ are completely determined by the number $N_{1,3d-1}$ plus those lower degree invariants using associativity.
\end{proposition}
\begin{remark}
Since there will be many complex terms popping up while playing the whole algebraic game with the equations (4.1) and (4.2), and the indices will be a mess. So, in the following proof, we decide to denote any lower degree terms or the combination of lower degree terms just by a single symbol $S$ meaning "some lower degree terms". Thus, \textbf{we cannot generally cancel $S$ out even if we see the symbol $S$ appears at the same time on the both sides of an equation}. 
\end{remark}
\begin{proof}
Note that the conclusion is true when $d=1$, see \S\ref{subsection 4.2}. Therefore, we can assume $d\geq2$.

\medskip

First of all, notice that, by the equation (4.1), all the invariants $N_{ab0}^d$ are actually determined by $N_{3d-1,1,0}^d$ and lower degree invariants. Indeed, set $r=1$, then $p+q=3d-1$, so by (4.1), we have equations 
$$N_{3d-1,1,0}^d+S=N_{p,q+1,0}^d+S.$$
Then for any $a,b$ such that $a+b=3d$, set $p=a$ and $q=b-1$, we can solve for $N_{ab0}^d$ by using $N_{3d-1,1,0}^d$ and those lower degree terms. \\

Secondly, in the equation (4.2), we set $r=i=1$, so $p+q=3d$, then we get 
\begin{equation}
\label{eq:first formula}
    (3d-1)N_{1,3d-1}+N_{pq0}^d+S=qN_{p,q}+(p-1)N_{q+1,p-1}+N_{q10}^{(q+1)/3}\delta_{p,1}+S.
\end{equation}
If we set $r=i=2$, we have 
\begin{equation}
\label{eq:second formula}
    (3d-2)N_{2,3d-2}+N_{pq0}^d+S=qN_{p,q}+(p-2)N_{q+2,p-2}+N_{q20}^{(q+2)/3}\delta_{p,2}+S.
\end{equation}
Then in \eqref{eq:first formula}, set $p=3d-1, q=1$, and by the lemma \ref{lem 4.1.1}, we get 
\begin{equation}
    (3d-1)N_{1,3d-1}+N_{3d-1,1,0}^d+S=(3d-1)^2N_{1,3d-1}+(3d-2)N_{2,3d-2}+S,
\end{equation}
and in \eqref{eq:second formula}, set $p=1, q=3d-1$, then we get
\begin{equation}
    (3d-2)N_{2,3d-2}+N_{1,3d-1,0}^d+S=(3d-1)N_{1,3d-1}+S.
\end{equation}
Obviously, the equations (4.5) and (4.6) are independent and we can solve for both $N_{3d-1,1,0}^d$ and $N_{2,3d-2}$ by using $N_{1,3d-1}$ (bearing in mind that $N_{ab0}^d=N_{ba0}^d$ for every $a,b$). Henceforth, all numbers $N_{ab0}^d$ can be solved by $N_{1,3d-1}$ and lower degree invariants, and $N_{2,3d-2}$ can be solved by $N_{1,3d-1}$ plus lower degree invariants.\\

Furthermore, (4.3) subtract (4.4) tells us that $(p-1)N_{q+1,p-1}-(p-2)N_{q+2,p-2}$ can be solved by using $N_{1,3d-1}$ and those lower degree invariants for any $p,q$ such that $p+q=3d$. Then, just by an easy algebra, we can conclude that all $N_{a,b}$ can be expressed by $N_{1,3d-1}$ and lower degree invariants. 
\end{proof}

\begin{remark}
The same conclusion should also hold for general $(\mathbb{P}^n, D)$ where $D$ is a smooth anti-canonical divisor, that is, all degree $d$, 2-pointed relative Gromov-Witten invariants with a point condition should be determined by $N_{(n+1)d-1,1}$ plus the lower degree invariants, and we can proceed the proof in an analogous fashion as we did in the proof above with a much more complex algebraic deduction. 
\end{remark}
 
\subsection{Calculations of degree 2 relative invariants}\label{subsection 4.2}

In this subsection, we are going to compute all the relevant degree 2 invariants. Before doing that, let us see a warm-up computation for degree $1$ invariants.\\

The first non-trivial case to consider is $(p,q,r)=(2,1,1)$, so $i=1$, and we get 
$$2N_{1,2}+N_{120}^1=2N_{2,1}.$$

If we take $(p,q,r)=(3,1,1)$, then we get 
$$4N_{1,2}=N_{2,1}.$$
It is well-known that the number of lines tangent to order 2 at a specified point of $E$ is $1$, in other words, $N_{1,2}=1$. So, $N_{2,1}=4$ and $N_{120}^1=6$.

\begin{proposition}\label{prop 4.2.1}
We have relations $25N_{1,5}=N_{5,1}, 2N_{2,4}=5N_{1,5}+2, N_{3,3}=5N_{1,5}+4$ and $N_{4,2}=10N_{1,5}+4.$
\end{proposition}
\begin{proof}
The way to get these relations is very straightforward. We just try to pick some special values for $(p,q,r)$ and plug them into equations (4.1) and (4.2), and then get a series of equations about these invariants. First of all, let $(p,q,r)=(1,2,3)$, by $(4.1)$, we have $N_{330}^2=N_{150}^2+24$; Let $(p,q,r)=(3,1,2)$, again by $(4.1)$, we have $N_{420}^2+12=N_{330}^2$. So we have an equation $N_{420}^2=N_{150}^2+12$.
\begin{itemize}
    \item Let $(p,q,r)=(4,2,1)$, we have 
    \begin{align}
        5N_{1,5}+N_{420}^2+8=2N_{4,2}+3N_{3,3}
    \end{align}
    \item Let $(p,q,r)=(2,1,4)$, we have 
    \begin{align}
        3N_{3,3}+2N_{4,2}=4N_{2,4}+N_{5,1}+16
    \end{align}
    \item $(p,q,r)=(5,2,1):$
    \begin{align}
        5N_{1,5}+4N_{2,4}+8=3N_{3,3}
    \end{align}
    \item let $(p,q,r)=(1,5,2)$ and $N_{420}^2=N_{150}^2+12$, we have 
    \begin{align}
        N_{420}^2=5N_{1,5}+N_{5,1}+12
    \end{align}
    
    \item $(p,q,r)=(4,1,3):$
    \begin{align}
        N_{5,1}+3N_{3,3}+4=4N_{4,2}
    \end{align}
\end{itemize}
It is fairly easy to check that these $5$ equations are linearly independent and any other equation in above invariants given by choosing $(p,q,r)$ which are different from above will be recovered by the 5 equations above. So, by an easy algebra, we can get the relations that we want to derive. 
\end{proof}
\begin{corollary}
We have $N_{1,5}=1, N_{5,1}=25, N_{2,4}=7/2, N_{4,2}=14$, $N_{3,3}=9, N_{240}^2=N_{420}^2=42, N_{150}^2=N_{510}^2=30$ and $N_{330}^2=54$.
\end{corollary}
\begin{proof}
By the corollary \ref{corollary 4.2.1} and computations of $n_{\beta+h}$ in \cite{Lau}, \cite{LLW} and \cite{GS}, we know that $N_{5,1}=25$. Then, just by some easy algebra. So, other numbers can be computed out by the equations shown in the proof of proposition \ref{prop 4.2.1}.
\end{proof}

\begin{remark}
As Theorem 1.1 in \cite{Lau} indicates, open Gromov-Witten invariants of a toric Calabi-Yau manifold are determined by Gross-Siebert slab functions. For the anti-canonical bundle of $\mathbb{P}^2$, the slab function is $$1+z^{(1,0,0,0)}+z^{(0,1,0,0)}+z^{(-1,-1,0,0)}-2t+5t^2-32t^3+286t^4-3038t^5+\dotso$$
as shown in \cite{GS}, $\S5$. Hence, in theory, we have succeeded in calculating degree $d$, 2-pointed, relative Gromov-Witten invariants with a point condition for $\mathbb{P}^2$ with an elliptic curve for any degree $d$. 
\end{remark}

DPMMS, Centre for Mathematical Sciences, Wilberforce Road,
Cambridge, CB3 0WB, UK

\medskip

\textit{Email Address:} {\tt yw515@cam.ac.uk}
\end{document}